\documentclass[11pt]{article}

\usepackage[utf8]{inputenc}
\DeclareUnicodeCharacter{03A9}{\ensuremath{\Omega}}
\usepackage[T1]{fontenc}
\usepackage{lmodern}
\usepackage{geometry}
\usepackage{amsmath, amsthm, amssymb, mathtools}
\usepackage{enumitem}
\usepackage{tikz-cd}
\usepackage{quiver}
\usepackage[colorlinks=true, linkcolor=blue, citecolor=red, urlcolor=blue]{hyperref}
\usepackage{xurl}

\theoremstyle{plain}
\newtheorem{theorem}{Theorem}[section]
\newtheorem{lemma}[theorem]{Lemma}
\newtheorem{corollary}[theorem]{Corollary}
\newtheorem{proposition}[theorem]{Proposition}

\theoremstyle{definition}
\newtheorem{definition}[theorem]{Definition}
\newtheorem{example}[theorem]{Example}
\newtheorem{examples}[theorem]{Examples}
\newtheorem{remark}[theorem]{Remark}

\numberwithin{equation}{section}

\newcommand{\A}{\mathcal{A}}
\newcommand{\C}{\mathcal{C}}
\newcommand{\D}{\mathcal{D}}
\newcommand{\E}{\mathcal{E}}
\newcommand{\M}{\mathcal{M}}
\newcommand{\K}{\mathcal{K}}
\newcommand{\V}{\mathcal{V}}
\newcommand{\N}{\mathbb{N}}
\newcommand{\Set}{\mathbf{Set}}
\newcommand{\CAT}{\mathbf{CAT}}
\newcommand{\ACC}{\mathbf{ACC}}
\newcommand{\Mon}{\mathbf{Mon}}
\newcommand{\Hopf}{\mathbf{Hopf}}
\newcommand{\Psd}{\mathbf{Psd}}
\newcommand{\Comon}{\mathbf{Comon}}
\newcommand{\Mod}{\mathbf{Mod}}
\newcommand{\Vect}{\mathbf{Vect}}
\newcommand{\Term}{\mathrm{Term}}
\newcommand{\Obj}{\mathrm{Obj}}
\newcommand{\Mor}{\mathrm{Mor}}
\newcommand{\coc}{\mathrm{coc}}
\newcommand{\op}{\mathrm{op}}
\newcommand{\Ins}{\mathrm{Ins}}
\newcommand{\Eq}{\mathrm{Eq}}
\newcommand{\LE}{\mathrm{LE}}
\newcommand{\RE}{\mathrm{RE}}
\newcommand{\Ran}{\mathrm{Ran}}
\newcommand{\id}{\mathrm{id}}
\newcommand{\B}{\mathbb{B}}
\newcommand{\Bialg}{\mathrm{Bialg}}
\newcommand{\BialgB}{\mathrm{Bialg}(\mathbb{B})}
\DeclareMathOperator*{\colim}{colim}

\author{David Forsman\\
Université catholique de Louvain\\
\texttt{david.forsman@uclouvain.be}}
\title{A Syntactic Approach to Ulmer's Bialgebras}
\date{24 November, 2026}

\begin{document}
\maketitle

\begin{abstract}
Ulmer proposed to view a bialgebraic structure on an object $x$ as a structure map $f\colon Fx\to Gx$ satisfying an equation $\alpha_{(x,f)} = \beta_{(x,f)}$ of natural transformations. This perspective is powerful enough to account for many algebraic and coalgebraic structures, such as varieties of universal algebra, categories of algebras over a PROP, an operad or a monad. We develop Ulmer's bialgebras within a $2$-categorical setting with PIE limits, where structure maps are abstracted by inserters and their equations by equifiers. A syntax of signatures, theories and models is then used to produce these bialgebraic structures. The syntax of the commutative diagrams defining Hopf algebras can be used, as is, to define the $0$-cell of Hopf algebras within any symmetric pseudomonoid in a $2$-category with PIE limits.

In this setting, we show that lax morphisms of models induce morphisms between the objects of algebras, extending the classical result lifting a lax monoidal functor to a functor between the categories of internal monoids. Since the construction of the object of bialgebras is expressed entirely in terms of PIE limits, we show that Kan extensions, accessibility, local presentability, orthogonal factorization systems, regularity, and Barr-exactness lift along the construction under suitable assumptions. A bi-dense class of objects allows us to transfer monadicity and local presentability results to formal settings. Finally, we construct a bi-dense class of $\V$-categories with at most three objects, where $\V$ is any monoidal category with an absorbing initial object.
\end{abstract}



\section{Introduction}

In his 1977 preprint \cite{Ulmer1977Bialgebras}, Ulmer introduced a semantic notion of a bialgebra to unify a broad range of algebraic and coalgebraic structures within the setting of locally presentable categories. Ulmer's insight was to view a bialgebraic structure on an object $c$ as a structural morphism $f\colon Fc\to Gc$ satisfying an equation $\alpha_{(c,f)} = \beta_{(c,f)}$ of natural transformations. Varieties of universal algebra, such as categories of groups or rings, Eilenberg--Moore categories of monads, internal bimonoids or Hopf algebras in symmetric monoidal categories are all of this form. Rather than treating these examples individually, Ulmer showed that many of their fundamental structural properties follow from a common semantic framework. Ulmer left out the syntactic or schematic approach, which he described as ``hopelessly involved'' \cite[3.1]{Ulmer1977Bialgebras}.

Our guiding idea is that the commutative diagrams defining a bialgebraic structure can be understood as syntax for equations between $2$-cells. For example, the commutative diagrams defining Hopf or Frobenius algebras in a symmetric monoidal category can be interpreted as pairs of natural transformations constructed from the data of a symmetric pseudomonoid in a $2$-category with PIE limits. This point of view gives meaning to $T$-bialgebras, for a bialgebraic theory $T$, in $2$-categories whose objects might not have any concrete morphisms of their own.

We develop Ulmer's notion of bialgebras within an abstract $2$-category $\K$ in two different layers. In the first semantic layer, we define a \emph{bialgebraic datum} on an object $c$ in $\K$ as a tuple
$$
\B = (c,\ f,g\colon c\to x,\ f',g'\colon c\to y,\ \alpha,\beta\colon f'u'\Rightarrow g'u')
$$
consisting of two \emph{function support} morphisms $f,g$, whose inserter $u'\colon \Ins(f,g)\to c$ defines Ulmer's \textit{object of pre-bialgebras} with the \textit{equifier} $e\colon \Eq(\alpha,\beta)\to \Ins(f,g)$ defining the object $\BialgB$ of \textit{bialgebras} of $\B$. The structural results of this paper mostly relate to such objects of bialgebras. The second layer is syntactic. A \textit{signature pair} $(\Sigma,\sigma)$ and a bialgebraic $\sigma$-\textit{theory} $T$ describe the structure maps and their equations relative to a $2$-algebraic structure fixed by $\Sigma$, such as a monoidal or a monadic structure. For a fixed $\Sigma$-model $M$ in $\K$, we can construct the bialgebraic datum $\B_{M,T}$, whose object of bialgebras is the object $M^T$ of \textit{internal $T$-bialgebras}.

\smallskip
\noindent\textbf{Main results:} Capital letters $F,G,F',G'$ indicate that $\K = \CAT$.
\begin{itemize}
    \item \textit{Induced Functor of Algebras} (Theorem \ref{thm:induced_morphism_between_algebras}). We define general notions of lax, oplax and strong morphisms $M\to N$ of $\Sigma$-models in $\K$, which induce a $1$-cell $M^T\to N^T$ for algebraic, coalgebraic and bialgebraic theories $T$, respectively. This generalizes the classical lifting of a lax monoidal functor to a functor between the categories of internal monoids.

    \item \textit{Strict creation of Kan extensions} (Theorem \ref{thm:pie_kan_creation}, Corollary \ref{cor:strictCreationToBialgebraicDatumObjectsKan}). We show that the canonical morphism $u_\B\colon \BialgB\to c$ strictly creates those left Kan extensions that the arity support $f$ preserves and whose quotient initiality $f'$ preserves. This generalizes the classical result that the forgetful functor $\C^T\to \C$ from the Eilenberg--Moore category of a monad $T$ on $\C$ strictly creates the colimiting cocones that $T$ preserves and whose joint epicness $T^2$ preserves.

    \item \textit{Monadicity and detection of left Kan extensions} (Theorem \ref{thm:BigMonadicityB}, Corollary \ref{cor:BigMonadicityModel}). If the coarity supports $G$ and $G'$ of $\B$ are right adjoints, $\C$ has colimits of size up to an infinite cardinal $\kappa$ and the arity supports $F$ and $F'$ preserve colimits of $\kappa$-chains, then $U\colon \BialgB\to \C$ is monadic and $\BialgB$ has colimits of size at most $\kappa$. These results are then upgraded to the $2$-categorical setting appropriately in Section \ref{sec:representable}.

    \item \textit{Local presentability} (Theorem \ref{thm:LPandComonadicity}). Over a locally presentable base $\C$, $\BialgB$ is locally presentable and $U_\B$ is monadic, given that the support functors are accessible and the coarity support functors $G$ and $G'$ are continuous. The corresponding comonadicity is obtained when the arity functors $F$ and $F'$ are cocontinuous. This theorem is generalized suitably to the $2$-categorical setting in Theorem \ref{thm:2CatLPComonadicity}. The definition of local presentability of an object in a $2$-category is given in Definition \ref{def:ALocalPresentability}.

    \item \textit{Factorization systems, regularity and exactness} (Theorems \ref{thm:pie_ofs_creation} and \ref{thm:liftingregularexact}). If the category $\C$ is equipped with an orthogonal factorization system $(\E,\M)$ such that $F(\E)\perp G(\M)$ and $F'(\E)\perp_! G'(\M)$, then $\BialgB$ has a unique orthogonal factorization system that the forgetful functor $U_\B\colon \BialgB\to \C$ preserves. Combining the limit and orthogonal factorization creation results, we obtain conditions under which regularity and exactness lift from the base category $\C$ to the object $\BialgB$ of bialgebras.
\end{itemize}

\smallskip
\noindent\textbf{Related work:}
Ulmer's work \cite{Ulmer1977Bialgebras} was set in the context of locally presentable categories, where he studied completeness and size properties of bialgebraic structures. Among the examples given were algebras over a PROP and descent data. Ulmer left open the local presentability of arbitrary Hopf algebras over a commutative ring \cite[4.6]{Ulmer1977Bialgebras}. Porst later proved this local presentability \cite{Porst-CoLimits} and studied local presentability of internal monoids in symmetric monoidal settings \cite{Porst01062008} and of Hopf algebras \cite{PorstHopfLocallyPresentable, Porst2}. We also touch on the setting of Hopf algebras by proving the exactness of the category $\Hopf_R^\coc$ of cocommutative Hopf algebras over a commutative von Neumann regular ring $R$ (Corollary \ref{cor:algebraicExtensionLiftingExactness}), extending the result for fields \cite{GRAN20194171}. The same corollary also establishes the exactness of many structures extending cocommutative Hopf algebras, such as $\Omega$-Hopf algebras \cite{bevilacquaCoalgebraicModelsOmega2026} and cocommutative Hopf braces \cite{granHopfBracesSemiabelian2026}.

Theorem \ref{thm:pie_ofs_creation} shows how to lift orthogonal factorization systems to categories of bialgebras, generalizing results by Linton \cite{Linton1969} and Wi{\ss}mann \cite{Wissmann2022Minimality}. Linton also showed in the same article how to compute colimits in the Eilenberg--Moore category $\C^T$ of a monad, provided that $\C^T$ has reflexive coequalizers. This result was extended by Barr \cite{barrCoequalizersFreeTriples1970} and Ad\'{a}mek \cite{adamekColimitsAlgebrasRevisited1977}. Theorem \ref{thm:BigMonadicityB} further develops these ideas by computing left Kan extensions without completeness and smallness assumptions, but requiring the arity functors' adherence to colimits of $\kappa$-chains.

On a $2$-categorical note, our objects of bialgebras are constructed from PIE limits, which are flexible limits \cite{birdFlexibleLimits2categories1989}. We prove $2$-categorical monadicity results (Theorems \ref{thm:2CategoricalMonadicityForBialgebras}, \ref{thm:2CatLPComonadicity}), which connect to the formal theory of monads \cite{STREET1972149, lackFormalTheoryMonads2002}. The notion of bi-density, which we use to generalize our results to objects of $2$-categories, is a variation of the bicategorical notion of density due to Hofstra and Karvonen \cite{hofstraInnerAutoequivalencesGeneral2024}. The notion of density was introduced by Isbell in \cite{isbellAdequateSubcategories1960}, further developed by Ulmer \cite{ulmerPropertiesDenseRelative1968} and Dubuc \cite{dubuc1970}.

\smallskip
\noindent\textbf{Outline:} Section~\ref{sec:preliminaries} reviews the required $2$-categorical preliminaries and defines strict products, inserters and equifiers in $2$-categories, so-called PIE limits. Section~\ref{sec:bialgebras} defines bialgebraic data $\B$ and their object $\BialgB$ of bialgebras, and develops signature pairs $(\Sigma,\sigma)$, theories $T$ and $\Sigma$-models $M$, constructs $M^T$, and proves the Induced Functor of Algebras Theorem. Section~\ref{sec:lifting_limits} proves the creation result for Kan extensions and specializes it to sifted colimits and wide pushouts of regular epimorphisms in $\CAT$. Section~\ref{sec:monadicity} studies the detection of colimits and proves monadicity theorems for the forgetful functor $\BialgB\to \C$, together with comonadicity under accessibility assumptions. Section~\ref{sec:ofs} proves the lifting result for orthogonal factorization systems along PIE limits and applies this result to derive corresponding lifting theorems for regularity and exactness. Finally, Section~\ref{sec:representable} uses representability and bi-density to generalize the monadicity and comonadicity results to $2$-categories. The appendix shows that the $\V$-categories $\textbf{1},\textbf{2}_x, \textbf{3}_{x,y}$ for objects $x,y\in \Obj(\V)$ form a bi-dense class of objects in $\V\text{-}\CAT$ when $\V$ is a monoidal category with an absorbing initial object.

\section{2-Categorical Preliminaries and PIE Limits}\label{sec:preliminaries}

Throughout this paper, we work in a strict $2$-categorical setting. We assume familiarity with the basics of $2$-categories, bicategories, and enriched categories, referring the reader to \cite{Kelly:Enriched, lack2CategoriesCompanion2010}.

\subsection{Conventions and Notation}
Let $\K$ be a strict $2$-category.\footnote{We assume the existence of a Grothendieck universe $U$ and its successor $U^+$. The elements of $U$ are called small sets, and elements of $U^+$ are large.}
\begin{itemize}
    \item Objects, morphisms ($1$-cells) and transformations ($2$-cells) of $\K$ are denoted as follows:\begin{itemize}
        \item \textbf{Objects:} $a,b,c,\ldots$
        \item \textbf{Morphisms:} $f,g,h,\ldots \colon a\to b$ are the objects of the hom-category $\K(a,b)$ for objects $a,b$ in $\K$.
        \item \textbf{Transformations:} $\alpha,\beta,\gamma\colon f\Rightarrow g\colon a\to b$ are the morphisms of $\K(a,b)$.
    \end{itemize}
    \item Vertical composition of $2$-cells in $\K(a,b)$ is denoted $\beta \bullet \alpha$, horizontal composition of $1$-cells is denoted $gf$ for morphisms $a\xrightarrow{f}b\xrightarrow{g} c$ in $\K$, and horizontal composition of $2$-cells is denoted $\beta * \alpha$ for transformations $\alpha\colon f_1\Rightarrow f_2\colon a\to b$ and $\beta\colon g_1\Rightarrow g_2\colon b\to c$ in $\K$. The bifunctoriality of horizontal composition can be seen as the interchange law $\beta * \alpha = \beta f_2 \bullet g_1\alpha = g_2\alpha \bullet \beta f_1$, with whiskerings defined by $\beta g \coloneqq \beta * \id_g$ and $h\alpha \coloneqq \id_h * \alpha$.
    \item For a monoidal category $\V$, $\V\text{-}\CAT$ denotes the strict $2$-category of $\V$-categories, $\V$-functors, and $\V$-natural transformations \cite{Kelly:Enriched}. In particular, $\CAT\text{-}\CAT$ is the $2$-category of $2$-categories, (strict) $2$-functors and $2$-natural transformations.
\end{itemize}

As the limits we work with are flexible, the results have a corresponding non-strict version for bicategories via strictification \cite{birdFlexibleLimits2categories1989, GordonPowerStreet1995CoherenceTricategories}.

\subsection{PIE Limits}
We recall the notions of strict PIE limits (products, inserters, and equifiers).
\begin{definition}[PIE Limits in a $2$-Category]
   Let $\K$ be a $2$-category. We define the following $2$-dimensional limits by their explicit universal properties on both $1$-cells and $2$-cells:
   \begin{itemize}
       \item \textbf{Products:} The product of a family of objects $(c_i)_{i \in I}$ is an object $\prod_i c_i$ equipped with projection morphisms $\pi_j \colon \prod_i c_i \to c_j$ satisfying:
       \begin{enumerate}[label=(\roman*)]
           \item For any family of morphisms $(f_i\colon x\to c_i)_{i\in I}$ in $\K$ there is a unique morphism $f\colon x\to \prod_i c_i$ such that $f_i = \pi_i f$ for each $i\in I$. We often denote $f = (f_i)_i$.
           \item For any family $(\alpha_i\colon f_i\Rightarrow g_i\colon x\to c_i)_{i\in I}$ of transformations in $\K$, there exists a unique transformation $\alpha\colon (f_i)_i\Rightarrow (g_i)_i\colon x\to \prod_i c_i$ such that $\alpha_i = \pi_i* \alpha$ for each $i\in I$, often denoted $\alpha = (\alpha_i)_i$.
       \end{enumerate}

       \item \textbf{Inserters:} Let $f,g\colon c\to d$ be parallel morphisms in $\K$. The \textbf{inserter} of $f$ and $g$ is an object $\Ins(f,g)$ in $\K$ equipped with a morphism $u\colon\Ins(f,g) \to c$ and a universal transformation $\phi \colon f \circ u \Rightarrow g \circ u$ satisfying:
       \begin{enumerate}[label=(\roman*)]
           \item For any morphism $h'\colon x\to c$ with a transformation $\alpha\colon fh'\Rightarrow gh'$, there exists a unique morphism $h\colon x\to \Ins(f,g)$ such that $h' = uh$ and $\alpha = \phi*h$.
           \item For any pair of morphisms $h_1,h_2\colon x\to \Ins(f,g)$ with a transformation $\alpha'\colon uh_1\Rightarrow uh_2$ making the diagram
           \begin{equation}\label{cd:2ndUnivInserter}
           \begin{tikzcd}
fuh_1 \arrow[d, "f*\alpha'"', Rightarrow] \arrow[r, "\phi*h_1", Rightarrow] & guh_1 \arrow[d, "g*\alpha'", Rightarrow] \\
fuh_2 \arrow[r, "\phi*h_2"', Rightarrow]                                   & guh_2
\end{tikzcd}
           \end{equation}
           commute, there is a unique transformation $\alpha\colon h_1\Rightarrow h_2$ such that $\alpha' = u*\alpha$.
       \end{enumerate}

       \item \textbf{Equifiers:} Let $\eta, \theta \colon f\Rightarrow g\colon c\to d$ be parallel transformations in $\K$. The \textbf{equifier} of $\eta$ and $\theta$ is an object $\Eq(\eta, \theta)$ equipped with a morphism $u \colon \Eq(\eta, \theta) \to c$ such that $\eta * u = \theta * u$, satisfying:
       \begin{enumerate}[label=(\roman*)]
           \item Any morphism $h'\colon x\to c$, satisfying $\eta*h' = \theta*h'$, has a unique lift $h\colon x\to \Eq(\eta,\theta)$ along $u$, meaning $h' = uh$.
           \item The functor $u_x\coloneqq u*(-)\colon \K(x,\Eq(\eta,\theta))\to \K(x,c)$ is fully faithful for every object $x$ in $\K$.
       \end{enumerate}
   \end{itemize}
   We say that $\K$ is \textbf{PIE-complete} if it has small products, inserters and equifiers.
\end{definition}

\begin{remark}[PIE Limits in $\CAT$]
   In the $2$-category $\CAT$, small products are standard cartesian products of categories, and inserters and equifiers are computed explicitly:
   \begin{itemize}
       \item Let $F, G \colon \C \to \D$ be parallel functors. The \textbf{inserter category} $\Ins(F, G)$ has as objects pairs $(c, \alpha\colon Fc\to Gc)$, where $c$ is an object of $\C$ and $\alpha$ is a morphism in $\D$. A morphism $f\colon (c,\alpha)\to (d,\beta)$ in $\Ins(F,G)$ consists of a morphism $f\colon c\to d$ in $\C$ making the diagram
       \[
       \begin{tikzcd}
Fc \arrow[d, "F(f)"'] \arrow[r, "\alpha"] & Gc \arrow[d, "G(f)"] \\
Fd \arrow[r, "\beta"']                             & Gd
       \end{tikzcd}
       \]
       commute. The universal functor $U \colon \Ins(F, G) \to \C$ is the forgetful functor $f\colon (c,\alpha)\to (d,\beta)\mapsto f\colon c\to d$, and the universal natural transformation $\phi$ has components $\phi_{(c,\alpha)} = \alpha$.

       \item Let $\eta,\theta\colon F\Rightarrow G\colon \C\to \D$ be natural transformations. The \textbf{equifier category} $\Eq(\eta, \theta)$ is the full subcategory of $\C$ consisting of those objects $c$ for which $\eta_c = \theta_c$. The universal functor $E \colon \Eq(\eta, \theta) \to \C$ is the canonical fully faithful inclusion.
   \end{itemize}
\end{remark}
The universal properties above say precisely that the representable $2$-functors $\K(x,-)\colon \K\to \CAT$ preserve and jointly reflect PIE limits, which in $\CAT$ are described in the remark above.
\begin{example}
    We give some examples of $2$-categories with PIE limits:
    \begin{enumerate}
        \item The $2$-category $\CAT$ of locally small categories has all PIE limits.
        \item The $2$-category $\ACC$ of accessible categories, accessible functors and natural transformations is closed under the PIE limits of $\CAT$ \cite{MakkaiPare1989AccessibleCategories, LPC}. Accessibility generalizes to two different notions for $\V\text{-}\CAT$, when $\V$ is a locally presentable symmetric monoidal closed category \cite{LACK2023107196}. For both of these notions of accessibility, the $2$-category $\V\text{-}\ACC$ of $\V$-accessible $\V$-categories, $\V$-accessible $\V$-functors and $\V$-natural transformations is closed under the PIE limits in $\V\text{-}\CAT$ \cite[Theorem 5.5, Theorem 5.9]{LACK2023107196}.
    \end{enumerate}
\end{example}

\section{Bialgebras}\label{sec:bialgebras}

Here we give the $2$-categorical definition of an object of bialgebras, extending Ulmer's definition \cite{Ulmer1977Bialgebras} from $\CAT$ to any $2$-category with PIE limits. The choice of structural maps of a bialgebraic object is captured by inserters and equations are enforced by equifiers. We then develop a multi-sorted, $2$-categorical syntactic framework (signature pairs and bialgebraic theories) which produces the data of this definition automatically from the syntax itself.

\subsection{Bialgebras in a \texorpdfstring{$2$}{2}-Category}

\begin{definition}[Bialgebraic Data and Their Bialgebras]\label{def:general_bialgebras}
   Let $\K$ be a $2$-category with inserters and equifiers, and let $c$ be an object in $\K$. A \textbf{bialgebraic datum} on $c$ is a tuple
   $$
   \mathbb{B} = \big(c,\ f,g\colon c\to x,\ \ f',g'\colon c\to y,\ \ \alpha,\beta\colon f'u'\Rightarrow g'u'\big),
   $$
   where $(u'\colon \Ins(f,g)\to c,\ \phi\colon fu'\Rightarrow gu')$ is the inserter of $f$ and $g$. The morphisms $f,g,f',g'$ are called the \textbf{supports} of $\mathbb{B}$:
   \begin{itemize}
       \item $f$ is the \textbf{function arity support} and $g$ the \textbf{function coarity support};
       \item $f'$ is the \textbf{equational arity support} and $g'$ the \textbf{equational coarity support}.
   \end{itemize}
   Thus $f,f'$ are the \textbf{arity supports}, $g,g'$ the \textbf{coarity supports}, $f,g$ the \textbf{function supports} and $f',g'$ the \textbf{equational supports}. The pair $(\alpha,\beta)$ is called the \textbf{equations} of $\mathbb{B}$.
   \begin{enumerate}
       \item The \textbf{object of pre-bialgebras} of $\mathbb{B}$ is the inserter $u'\colon \Ins(f,g)\to c$.
       \item The \textbf{object of bialgebras} of $\mathbb{B}$ is the equifier
       $$
       \mathrm{Bialg}(\mathbb{B})\coloneqq \Eq(\alpha,\beta)\xrightarrow{e}\Ins(f,g).
       $$
       \item By $u_\B = u\colon\BialgB\to c$ we denote the composite $u\coloneqq u'e$.
   \end{enumerate}
   We use capital letters $(\C, F,G,\ F',G',\ \alpha,\beta)$ to indicate the concrete nature of $\K$, for instance, when $\K = \CAT$ or $\K =\V$-$\CAT$.
\end{definition}

Originally, Ulmer considered families of pairs of support morphisms $(f_\mu,g_\mu\colon c\to x_\mu)_{\mu}$ and equations $(\alpha_r,\beta_r)_{r}$ indexed by classes. When the classes are sets, this data can be turned into one pair by taking products, $f=(f_\mu)_\mu, g= (g_\mu)_\mu\colon c\to\prod_\mu x_\mu$, and similarly for the equations. For $\K=\CAT$ this recovers Ulmer's definition \cite{Ulmer1977Bialgebras} when the indexing families are small. We sidestep the size issues via the Grothendieck universe approach.

The syntax devised below produces bialgebraic data automatically. We develop the notions of a signature pair $(\Sigma,\sigma)$, a $\sigma$-theory $T$ and a $\Sigma$-model $M$, which together determine a bialgebraic datum on $M_0$. The associated object of bialgebras is then $M^T$ (Definition~\ref{def:objectOfTBialgebras}).

\begin{remark}\label{rmk:AdjointsInducingIsosBetweenBialgebras}
    Consider a bialgebraic datum $\B = (c,\ f,g,\ f',g',\ \alpha,\beta)$ in $\K$ with $g,g'$ having left adjoints $l$ and $l'$, respectively. Note that $\BialgB$ is isomorphic to $\Bialg(\B')$ over $c$, where $\B' = (c,\ c\xrightarrow{lf,\id_c} c,\ c\xrightarrow{l'f',\id_c}c,\ l'f'u''\xRightarrow{\alpha',\beta'} u''\colon \Ins(lf,\id_c)\to c)$ and $\alpha'$ and $\beta'$ are transformations induced by the adjunctions: Notice first that the adjunction $l\dashv g$ induces an isomorphism $\Ins(f,g)\cong \Ins(lf,\id_c)$ over $c$. This allows one to consider $\alpha'',\beta''\colon f'u''\Rightarrow g'u''$, where $u''\colon \Ins(lf,\id_c)\to c$ is the inserter, and whose mates we denote by $\alpha',\beta'\colon l'f'u''\Rightarrow u''$.

    The previous construction also works in the case that $g$ and $g'$ only have $f$- and $f'$-relative left adjoints $r,r'\colon c\to c$, respectively.\footnote{If $F\colon I\to \D$, $J\colon I\to \C$ and $G\colon \D\to \C$ are functors, then we say that $F$ is a $J$-relative left adjoint of $G$ if there is an isomorphism $\D(Fx,y)\cong \C(Jx,Gy)$ natural in $x$ and $y$. One can then define when $f$ is a $j$-relative left adjoint of $g$ in a $2$-category $\K$ via representable $2$-functors.}
\end{remark}

\subsection{Syntax: Signature Pairs and Theories}
To motivate the syntactic constructions in this section, consider the syntax needed to formulate internal monoids in a lax monoidal setting:
\[
\Sigma = (c, \quad cc\xrightarrow{\square} c,\ ()\xrightarrow{I} c, \quad (x\square y)\square z\xrightarrow{\alpha^{xyz}}x\square (y\square z),\ x\square I\xrightarrow{\rho^x} x,\ I\square x\xrightarrow{\lambda^x} x),
\]
a relative $1$-signature describing the monoid's multiplication and unit operations
\[
\sigma = (s\colon c, \quad \mu\colon s\square s\to s, \quad u\colon I\to s),
\]
and the associativity and unit equations $T$. This information is packaged within a signature pair $(\Sigma,\sigma)$ and a $\sigma$-theory $T$.

Given a function $f \colon X \to Y$ between sets, we say that the set $X$ is \emph{typed} in $Y$ and write $x \colon f(x)$ for each element $x$ in $X$.

\begin{definition}[The $2$-Signature $\Sigma$]
   A \textbf{$2$-signature} $\Sigma = (C,\mathcal{F}, \mathcal{T})$ consists of the following data:
   \begin{enumerate}
       \item A small set $C$ of \textbf{category symbols}, often denoted by $c,d,c_1,c_2,\dots$.
       \item A typed small set $\mathcal{F} \to C^* \times C$ of \textbf{functor symbols} $B \colon c_1 \cdots c_n \to d$, where $C^*$ is the free monoid over $C$.
       \item A typed set $V \to C$, with countably infinite fibers, of \textbf{variable symbols} $x,y,z,\dots$. From these variables, we recursively construct a typed set $\Term^\Sigma_V \to C$ of \textbf{$\Sigma$-terms}.\footnote{A term $t\in \Term_V^\Sigma$ is either a variable $x\colon c$ in $V$, a constant $B\colon ()\to c$ (functor symbol with empty domain) or $t = B(t_0,\ldots, t_n)\colon c$ for some previously constructed terms $t_0\colon c_0,\ldots, t_n\colon c_n$ and a functor symbol $B\colon c_0\cdots c_n\to c\in \mathcal{F}$, where $n\in \N$ (with $0\in \N$).} The set of \textbf{contexts} $\mathrm{Cont}_V \subset V^*$ consists of sequences where no variable occurs more than once. A term $t \in \Term_V^\Sigma$ is in a context $\bar{x} \in \mathrm{Cont}_V$ if all variables appearing in $t$ occur in the sequence $\bar{x}$.
       \item A small set $\mathcal{T} \to \mathrm{Cont}_V \times \Term^\Sigma_V \times \Term^\Sigma_V$ of \textbf{transformation symbols} $\alpha^{\bar{x}} \colon t_1 \to t_2$, where $t_1, t_2$ are $\Sigma$-terms of the same type in context $\bar{x}$.
   \end{enumerate}
\end{definition}
\begin{example}\label{ex:2catInducesSignature}
Each small $2$-category $\A$ induces a $2$-signature $\Sigma_\A = (C, \mathcal{F}, \mathcal{T})$ as follows:
\begin{itemize}
    \item \textbf{Category symbols:} The set $C \coloneqq \Obj(\A)$. For each $a \in \Obj(\A)$, we fix a variable symbol $x_a$ of type $a$.
    \item \textbf{Functor symbols:} Each $1$-cell $f \colon a \to b$ in $\A$ gives a functor symbol $f \colon a \to b$.
    \item \textbf{Transformation symbols:} We include three families of transformation symbols:
    \begin{itemize}
        \item A transformation symbol $\eta^{x_a} \colon f(x_a) \to g(x_a)$ for each $2$-cell $\eta \colon f \Rightarrow g \colon a \to b$ in $\A$.
        \item A transformation symbol $\alpha^{x_a}_{g,f} \colon (gf)(x_a) \to g(f(x_a))$ for each composable pair $a \xrightarrow{f} b \xrightarrow{g} c$ in $\A$.
        \item A transformation symbol $\beta^{x_a}_a \colon \id_a(x_a) \to x_a$ for each object $a \in \Obj(\A)$.
    \end{itemize}
\end{itemize}
\end{example}

We define the internal algebraic structures on top of this background categorical universe by introducing a relative $1$-signature.

\begin{definition}[The $1$-Signature $\sigma$]
   A pair $(\Sigma,\sigma)$ is called a \textbf{signature pair} if $\Sigma = (C,\mathcal{F},\mathcal{T})$ is a $2$-signature and the \textbf{$1$-signature} $\sigma = (S, \Phi)$ relative to $\Sigma$ consists of:
   \begin{enumerate}
       \item A typed small set of \textbf{sorts} $S \to C$, assigning each sort $s \in S$ a type $c_s \in C$.
       \item A typed small set of \textbf{function symbols} $\Phi \to \Term_{S}^\Sigma \times \Term_{S}^\Sigma$. Each $f \in \Phi$ is associated with a tuple $(t_1,t_2)$, denoted $f\colon t_1 \to t_2 \colon c_f$, where $t_1, t_2\in \Term_S^\Sigma$ are terms of type $c_f$.
   \end{enumerate}
\end{definition}

We recursively generate a graph of terms in order to formulate equations between these operations.

\begin{definition}[The Term Graph]
   Let $(\Sigma,\sigma)$ be a signature pair. We construct the \textbf{term graph} $G_\sigma$ in the following recursive way:
   \begin{itemize}
       \item The set of vertices of $G_\sigma$ is the set $\Term^\Sigma_S$ of $\Sigma$-terms over the set of sorts $S$ of $\sigma$.
       \item An arrow of $G_\sigma$ is recursively defined by:
       \begin{itemize}
           \item Each function symbol $f \colon t_1 \to t_2$ of $\sigma$ forms an arrow $t_1 \to t_2$ in $G_\sigma$.
           \item For each transformation symbol $\alpha^{\bar{x}} \colon t_1 \to t_2$ in $\Sigma$, where $\bar{x} = x_1\cdots x_n$, we set $\alpha_{t_1',\dots, t_n'} \colon t_1^{\bar{x}}(t_1',\dots, t_n') \to t_2^{\bar{x}}(t_1',\dots, t_n')$ as an arrow in $G_\sigma$ for each sequence of terms $t_i' \in \Term_S^\Sigma$ of type $c_{x_i}$.\footnote{The term $t_i^{\bar{x}}(t_1',\ldots, t_n')$ refers to recursively defined substitution, where $t_j'$ is substituted for $x_j$ in $t_i$ for $i = 1,2$ and $j\leq n$.}
           \item For any functor symbol $B \colon c_0\cdots c_n \to d$ in $\Sigma$, a sequence of terms $t_j \colon c_j$, and an arrow $g \colon t_i \to t_i'$ in $G_\sigma$, the symbol sequence
           \[
           B(t_0,\dots, t_{i-1}, g, t_{i+1}, \dots, t_n) \colon B(t_0,\dots, t_i, \dots, t_n) \to B(t_0,\dots, t_i', \dots, t_n)
           \]
           defines an arrow in $G_\sigma$.\footnote{We use the indexing starting from $0$ in terms $B(t_0,\ldots, t_n)$ to indicate that $B$ is not a constant symbol. If the indexing starts from $1$, it implies the possibility that $B(t_1,\ldots, t_n)$ is a constant symbol with the choice $n = 0$.}
       \end{itemize}
   \end{itemize}
\end{definition}

Using this term graph, we define equations and classify theories based on their arity and coarity behaviors.

\begin{definition}[Arity, Coarity, and Bialgebraic Theories]
   Let $(\Sigma,\sigma)$ be a signature pair. A parallel path $(p, q \colon t_1 \to t_2\colon c)$ in the term graph $G_\sigma$ is called a \textbf{$\sigma$-equation}, and a set $T$ of $\sigma$-equations is called a \textbf{bialgebraic $\sigma$-theory}. We define the arity and coarity terms associated with $T$ as follows:
   \begin{enumerate}
       \item \textbf{Function Arity and Coarity:} For each function symbol $f \colon t_1 \to t_2$ in the $1$-signature $\sigma$, the domain term $t_1$ is a \textbf{function arity term} and the codomain term $t_2$ is a \textbf{function coarity term}.
       \item \textbf{Equational Arity and Coarity:} For each equation $(p, q \colon t_1' \to t_2')$ in the theory $T$, the domain term $t_1'$ is an \textbf{equational arity term} and the codomain term $t_2'$ is an \textbf{equational coarity term}.
       \item \textbf{Arity Terms:} The collection of all function arity terms and equational arity terms is the set of \textbf{arity terms}.
       \item \textbf{Coarity Terms:} The collection of all function coarity terms and equational coarity terms is the set of \textbf{coarity terms}.
   \end{enumerate}
   A bialgebraic $\sigma$-theory $T$ is called \textbf{algebraic} if all coarity terms consist solely of sorts $s \in S$. Similarly, the theory $T$ is called \textbf{coalgebraic} if all arity terms are sorts.
\end{definition}

\subsection{Semantics: Models and the Object of Bialgebras}

We now define the semantics of these signatures by interpreting them as models, operations, and equations within an arbitrary PIE-complete $2$-category. We fix a $2$-category $\K$ with finite products and a signature pair $(\Sigma, \sigma) = ((C,\mathcal{F}, \mathcal{T}),(S,\Phi))$ with a $\sigma$-theory $T$.

\begin{definition}[$\Sigma$-Models]
   A \textbf{$\Sigma$-model} $M$ in $\K$ assigns:
   \begin{enumerate}
       \item An object $M_c$ in $\K$ to each category symbol $c \in C$. For any typed sequence $\bar{x} = (x_1 \colon c_1, \dots, x_n \colon c_n)$ with $c_1,\ldots, c_n\in C$, we define the product object $M_{\bar{x}} = \prod_{i=1}^n M_{c_i}$ in $\K$.
       \item A $1$-cell $M(B) \colon M_{\bar{c}} \to M_d$ in $\K$ for each functor symbol $B \colon \bar{c} \to d \in \mathcal{F}$. For any $\Sigma$-term $t \colon c$ in a context $\bar{x}$, we recursively construct the term $1$-cell $M_{\bar{x}}(t) \colon M_{\bar{x}} \to M_c$.\footnote{Consider $t\in \Term^\Sigma_V$ in context $\bar{x}$. If $t = x_i$, then $M_{\bar{x}}(t) = \pi_{i}\colon M_{\bar{x}}\to M_{c_i}$ and if $t = B(t_1,\ldots, t_n)$, then $M_{\bar{x}}(t) = M(B)(M_{\bar{x}}(t_1),\ldots, M_{\bar{x}}(t_n))$.}
       \item A $2$-cell $M(\alpha) \colon M_{\bar{x}}(t_1) \Rightarrow M_{\bar{x}}(t_2)$ in $\K$ for each transformation symbol $\alpha^{\bar{x}} \colon t_1 \to t_2 \in \mathcal{T}$.
   \end{enumerate}
\end{definition}
\begin{example}\label{ex:2FunctorInucingAModel}
Let $\A$ and $\K$ be $2$-categories with $\A$ small, and let $\Sigma_\A$ be the $2$-signature of Example~\ref{ex:2catInducesSignature}. Each strict $2$-functor $F \colon \A \to \K$ induces a $\Sigma_\A$-model $M_F$ in $\K$, defined by:
\[
M_F(a) = Fa, \quad M_F(f) = Ff, \quad M_F(\eta^{x_a}) = F\eta, \quad M_F(\alpha_{g,f}^{x_a}) = \id_{F(gf)}, \quad M_F(\beta^{x_a}_a) = \id_{\id_{Fa}}.
\]
\end{example}

Between different $\Sigma$-models, we can define three natural notions of morphism.

\begin{definition}[Lax, Oplax, and Strong Morphisms]\label{def:laxStrongMorphisms}
   Let $M$ and $N$ be $\Sigma$-models in $\K$. A \textbf{lax morphism} $(F,\eta) \colon M \to N$ consists of the following data:
   \begin{enumerate}
       \item A $1$-cell $F_c \colon M_c \to N_c$ for each category symbol $c$ in $\Sigma$. We define $F_{\bar{x}} = \prod_{i\leq n} F_{c_i}\colon M_{\bar{x}}\to N_{\bar{x}}$ for every typed sequence $\bar{x} = (x_1\colon c_1,\ldots, x_n\colon c_n)$, where $c_1,\ldots, c_n$ are category symbols of $\Sigma$.
       \item A $2$-cell $\eta^B \colon N(B) \circ F_{\bar{c}} \Rightarrow F_d \circ M(B)$ for every functor symbol $B \colon \bar{c} \to d$ in $\Sigma$. We recursively extend this to define a $2$-cell $\eta_{t, \bar{x}} \colon N_{\bar{x}}(t) \circ F_{\bar{x}} \Rightarrow F_c \circ M_{\bar{x}} (t)$ for every $\Sigma$-term $t \colon c$ in a context $\bar{x} = x_1\cdots x_n$ as follows:
       {\small
   \begin{equation}\label{eq:fromPathToATransformation}
       \eta_{t,\bar{x}} = \begin{cases}
   \id\colon \pi_i\circ F_{\bar{x}}\Rightarrow F_{c_i}\circ \pi_i, & \text{if $t = x_i$ for $i\leq n$} \\
   N(B)(N_{\bar{x}}(t_1)F_{\bar{x}},\ldots, N_{\bar{x}}(t_n)F_{\bar{x}}) \\
\quad \xRightarrow{N(B)*(\eta_{t_1,\bar{x}},\dots, \eta_{t_n, \bar{x}})} N(B)F_{\bar{c}}(M_{\bar{x}}(t_1),\dots, M_{\bar{x}}(t_n))\\
   \quad \xRightarrow{\eta^B(M_{\bar{x}}(t_1),\dots, M_{\bar{x}}(t_n))} F_c M(B)(M_{\bar{x}}(t_1),\dots, M_{\bar{x}}(t_n)) & \text{if $t = B(t_1,\dots, t_n)$}
\end{cases}
\end{equation}
}
       \item For every transformation symbol $\theta^{\bar{x}} \colon t_1 \to t_2$ of type $c$ in $\Sigma$, we require that the diagram
       $$
       \begin{tikzcd}[ampersand replacement=\&]
       N_{\bar{x}}(t_1) \circ F_{\bar{x}} \arrow[r, "\eta_{t_1, \bar{x}}", Rightarrow] \arrow[d, "N(\theta)*F_{\bar{x}}"', Rightarrow] \& F_c \circ M_{\bar{x}}(t_1) \arrow[d, "F_c*M(\theta)", Rightarrow] \\
       N_{\bar{x}}(t_2) \circ F_{\bar{x}} \arrow[r, "\eta_{t_2, \bar{x}}"' , Rightarrow]                                               \& F_c \circ M_{\bar{x}}(t_2)
       \end{tikzcd}
       $$
       commutes.
   \end{enumerate}
   Dually, we define an \textbf{oplax morphism} by reversing the direction of the $2$-cells $\eta^B$. We say that $(F,\eta)$ is a \textbf{strong morphism} if it is a lax morphism where each $2$-cell $\eta^B$ is an isomorphism. A lax morphism $(F,\eta)\colon M\to N$ is called \textbf{strict} if $\eta^B$ is an identity for each $\Sigma$-functor symbol $B$.
\end{definition}

\begin{proposition}
   Let $(F,\eta) \colon M \to N$ and $(G,\theta) \colon N \to P$ be lax morphisms of $\Sigma$-models in $\K$. Then the composite $G \circ F$ is a lax morphism. Furthermore, oplax and strong morphisms are closed under composition.
\end{proposition}
\begin{proof}
   We sketch the case for lax morphisms. We define the components $(G \circ F)_c = G_c \circ F_c$ and define the structural $2$-cell $\phi^B$ for each functor symbol $B \colon \bar{c} \to d$ as the composite:
   $$
   \phi^B \colon P(B) \circ (GF)_{\bar{c}} \xRightarrow{\theta^B*F_{\bar{c}}} G_d \circ N(B) \circ F_{\bar{c}} \xRightarrow{G_d * \eta^B} (GF)_d \circ M(B)
   $$
   Via induction, for any $\Sigma$-term $t \colon c$ in a context $\bar{x}$, we obtain the factorization:
   \[
   \phi_{t, \bar{x}} \colon P_{\bar{x}}(t) \circ (GF)_{\bar{x}} \xRightarrow{\theta_{t, \bar{x}}*F_{\bar{x}}} G_c \circ N_{\bar{x}}(t) \circ F_{\bar{x}} \xRightarrow{G_c*\eta_{t, \bar{x}}} (GF)_c \circ M_{\bar{x}}(t)
   \]
   Consider a transformation symbol $\alpha^{\bar{x}} \colon t_1 \to t_2\colon c$ in $\Sigma$. Now the diagram
   \[
   \begin{tikzcd}[column sep = 2cm, ampersand replacement=\&]
   P_{\bar{x}}(t_1) \circ (GF)_{\bar{x}} \arrow[d, "P(\alpha)*(GF)_{\bar{x}}" description, Rightarrow] \arrow[rr, "\phi_{t_1, \bar{x}}", Rightarrow, bend left] \arrow[r, "\theta_{t_1, \bar{x}}*F_{\bar{x}}" description, Rightarrow] \& G_c \circ N_{\bar{x}}(t_1) \circ F_{\bar{x}} \arrow[r, "G_c*\eta_{t_1, \bar{x}}" description, Rightarrow] \arrow[d, "G_c*N(\alpha)*F_{\bar{x}}" description, Rightarrow] \& (GF)_c \circ M_{\bar{x}}(t_1) \arrow[d, "(GF)_c*M(\alpha)" description, Rightarrow] \\
   P_{\bar{x}}(t_2) \circ (GF)_{\bar{x}} \arrow[r, "\theta_{t_2, \bar{x}}*F_{\bar{x}}"' description, Rightarrow] \arrow[rr, "\phi_{t_2, \bar{x}}"', Rightarrow, bend right]                                        \& G_c \circ N_{\bar{x}}(t_2) \circ F_{\bar{x}} \arrow[r, "G_c*\eta_{t_2, \bar{x}}"' description, Rightarrow]                                                                \& (GF)_c \circ M_{\bar{x}}(t_2)
   \end{tikzcd}
   \]
   commutes. Thus, the composite $G \circ F$ is a lax morphism of $\Sigma$-models.
\end{proof}

To obtain the $2$-category of $\Sigma$-models and lax morphisms, and its sub-$2$-category of strong morphisms, we define the $2$-cells. A \textbf{transformation} $\gamma \colon (F,\eta_F) \Rightarrow (G, \eta_G) \colon M \to N$ of lax morphisms consists of $2$-cells $\gamma_c \colon F_c \Rightarrow G_c$ in $\K$ such that, writing $\gamma_{\bar{c}}\coloneqq \prod_i \gamma_{c_i}$, the diagram
\[
\begin{tikzcd}[ampersand replacement=\&]
N(B) \circ F_{\bar{c}} \arrow[d, "N(B)*\gamma_{\bar{c}}" description, Rightarrow] \arrow[r, "\eta^B_F", Rightarrow] \& F_d \circ M(B) \arrow[d, "\gamma_d*M(B)", Rightarrow] \\
N(B) \circ G_{\bar{c}} \arrow[r, "\eta_G^B"', Rightarrow]                                                           \& G_d \circ M(B)
\end{tikzcd}
\]
commutes for each functor symbol $B\colon \bar{c}\to d$ in $\Sigma$. Similarly, one obtains the $2$-category of $\Sigma$-models and oplax morphisms.

\begin{example}\label{ex:modificationsAndTransformations}
Let $F, G \colon \A \to \K$ be strict $2$-functors. We define the categories
\[
[F, G]_\mathbf{Lax}, \quad [F, G]_\mathbf{Oplax}, \quad [F, G]_\Psd, \quad [F, G]_\mathbf{Strict}
\]
as the categories whose objects are lax, oplax, strong, and strict morphisms $M_F \to M_G$ of $\Sigma_\A$-models in $\K$ (Example \ref{ex:2FunctorInucingAModel}), respectively, and whose morphisms are transformations between them. (The implicit size issues are sidestepped with Grothendieck universes.)

These objects correspond precisely to lax, oplax, pseudo-, and strict natural transformations $F \Rightarrow G$, and the morphisms between them are modifications.
\end{example}

\begin{definition}[The Object of Operations $M^\sigma$]
   Let $(\Sigma,\sigma)$ be a signature pair with set of sorts $S$, and let $M$ be a $\Sigma$-model in a $2$-category $\K$ with small products and inserters. We define $M_0 \coloneqq \prod_{s \in S} M_{c_s}$ and the \textbf{object $M^\sigma$ of operations} as the inserter
   \[
   \begin{tikzcd}
M^\sigma \arrow[d, "u_M"'] \arrow[r, "u_M"]                                    & M_0 \arrow[d, "(M(t_2))_f"]            \\
M_0 \arrow[r, "(M(t_1))_f"'] \arrow[ru, "\phi_M" description, Rightarrow] & \prod_{f} M_{c_f}
\end{tikzcd}
   \]
   where the indexing is over $\sigma$-function symbols $f\colon t_1\to t_2\colon c$.\footnote{For a term $t\in\Term_S^\Sigma$, we define $M(t)\colon M_0\to M_c$ recursively: $M(t) = \pi_s\colon M_0\to M_{c_s}$ if $t = s$ is a sort and $M(t) = M(B)(M(t_1),\ldots, M(t_n))$ if $t = B(t_1,\ldots, t_n)$.} For a term $t\colon c\in \Term_S^\Sigma$, we write $M^\sigma(t) \coloneqq M(t)\circ u_M\colon M^\sigma\to M_c$. If $\K = \CAT$, we call an object of $M^\sigma$ a \textbf{$\sigma$-model}.
\end{definition}
Explicitly, a $\sigma$-model $m$ in $M$ consists of an object $m_s$ of $M_{c_s}$ for each sort $s$ and a morphism $m(f)\colon M(t_1)((m_s)_s)\to M(t_2)((m_s)_s)$ for each function symbol $f\colon t_1\to t_2$.

\begin{definition}[Path Evaluation and the Object of Bialgebras $M^T$]\label{def:objectOfTBialgebras}
   Let $(\Sigma,\sigma)$ be a signature pair with a $\Sigma$-model $M$ in a PIE-complete $2$-category $\K$. We define a $2$-cell $\overline{p}_M \colon M^\sigma(t_1) \Rightarrow M^\sigma(t_2) \colon M^\sigma \to M_c$ recursively for each arrow $p \colon t_1 \to t_2 \colon c$ in the term graph $G_\sigma$:
   {\footnotesize
   \[
   \overline{p}_M = \begin{cases}
       \pi_f * \phi_M \colon M^\sigma(t_1) \Rightarrow M^\sigma(t_2), & \text{if $p = f \colon t_1 \to t_2$ is a function symbol in $\sigma$,} \\
       M(\alpha) * (M^\sigma(t_1'), \dots, M^{\sigma}(t_n')), & \text{if $p = \alpha_{t_1', \dots, t_n'}$ for a $\Sigma$-transformation symbol $\alpha^{\bar{x}}\colon t_1\to t_2$,} \\
       M(B) * (M^\sigma(t_1'), \dots, \bar{q}_M, \dots, M^\sigma(t_n')), & \text{if $p = B(t_1', \dots, q, \dots, t_n')$ for a $\Sigma$-functor symbol $B$.}
   \end{cases}
   \]
   }
   This extends to all paths $p$ in $G_\sigma$ by vertical composition; the empty path at a term $t$ is evaluated to the identity $2$-cell on $M^\sigma(t)$. The upcoming Lemma \ref{lem:substitution_evaluated} justifies that the definition of $\overline{p}_M$ is properly given by showing that the domains and codomains in the second recursive step agree, allowing the horizontal composition.

   For a bialgebraic $\sigma$-theory $T$, the \textbf{object of $T$-bialgebras} $M^T$ is defined via the equifier
   \[
       \begin{tikzcd}[ampersand replacement=\&]
{M^T} \& {M^\sigma} \&\& {\prod_{(p,q)} M_c}
\arrow["{e_M}", from=1-1, to=1-2]
\arrow[""{name=0, anchor=center, inner sep=0}, "{(M^\sigma(t_1))}", shift left=3, from=1-2, to=1-4]
\arrow[""{name=1, anchor=center, inner sep=0}, "{(M^\sigma(t_2))}"', shift right=3, from=1-2, to=1-4]
\arrow["{(\bar{p}_M)}"', shift right=3, Rightarrow, from=0, to=1]
\arrow["{(\bar{q}_M)}", shift left=3, Rightarrow, from=0, to=1]
\end{tikzcd}
   \]
   where the indexing is over $(p,q\colon t_1\to t_2\colon c)\in T$. If $\K = \CAT$, we call an object of $M^T$ a \textbf{$T$-bialgebra}.
\end{definition}

\begin{proposition}\label{prp:bialgebraic_datum}
    Let $M$ be a $\Sigma$-model in a PIE-complete $\K$. Then $M^\sigma$ is the object of pre-bialgebras and $M^T$ the object of bialgebras of $\B_{M,T}$ over $M_0$. The four support morphisms of $\B_{M,T} = (M_0,\ f,g,\ f',g',\ \alpha,\beta)$ are obtained from morphisms $M(t)\colon M_0\to M_c$, where $t\colon c$ runs over the function arity, function coarity, equational arity, and equational coarity terms, respectively, and
    $$
    \alpha = (\overline{p}_M)_{(p,q)}, \beta = (\overline{q}_M)_{(p,q)}\colon (M^\sigma(t_1))_{(p,q)}\Rightarrow (M^\sigma(t_2))_{(p,q)}\colon M^\sigma\to \prod_{(p,q)}M_c,
    $$
    where the indexing runs over $(p,q\colon t_1\to t_2\colon c)\in T$.
\end{proposition}
\begin{proof}
    This is just a retelling of the definition.
\end{proof}

\begin{definition}\label{def:supportsOfTerms}
We call the morphism $M(t)\colon M_0\to M_c$ a \textbf{function/equational arity/coarity support} if $t\colon c\in \Term_S^\Sigma$ is a function/equational arity/coarity term, respectively, relative to $T$.
\end{definition}

\begin{remark}\label{rem:UPEM}
    By the $2$-categorical Yoneda lemma, $M^T$ is characterized up to isomorphism by an isomorphism of categories $\K(x,M^T)\cong\K(x,M)^T$, natural in $x$, where $\K(x,M)$ is the $\Sigma$-model in $\CAT$ induced by $M$. For example, let $(t\colon c\to c,m\colon tt\Rightarrow t,e\colon \id_c\Rightarrow t)$ be a monad in $\K$. Its (strict) Eilenberg--Moore object \cite{STREET1972149} is a triple $(d,u\colon d\to c,\eta\colon tu\Rightarrow u)$ satisfying $\eta\bullet mu=\eta\bullet t\eta$ and $\eta\bullet eu=\id_u$, such that:
    \begin{enumerate}
        \item for every triple $(x,h'\colon x\to c,\psi\colon th'\Rightarrow h')$ satisfying the same two equations, there is a unique $h\colon x\to d$ with $h' = uh$ and $\psi = \eta h$;
        \item for all $h_1,h_2\colon x\to d$ and every $\beta\colon uh_1\Rightarrow uh_2$ with $\eta h_2\bullet t\beta=\beta\bullet\eta h_1$, there is a unique $\theta\colon h_1\Rightarrow h_2$ with $u\theta=\beta$.
    \end{enumerate}
    The universal property of a general $M^T$ admits an analogous explicit description.
\end{remark}

\begin{lemma}\label{lem:substitution_evaluated}
   Let $(F,\eta) \colon M \to N$ be a lax morphism of $\Sigma$-models in $\K$. Let us denote by $V$ and $S$ the sets of variable symbols of $\Sigma$ and sorts of $\sigma$, respectively. Let $t\in \Term_V^\Sigma$ be a term of type $c$ in a context $\bar{x} = x_1\cdots x_n$, where $x_i\colon c_i$ for $i\leq n$. Let $t_1\colon c_1,\ldots, t_n\colon c_n\in \Term_S^\Sigma$. Then we have the following factorizations:
   \begin{align}
           M(t^{\bar{x}}(t_1,\ldots, t_n)) &\colon M_0\xrightarrow{(M(t_1),\ldots, M(t_n))}M_{\bar{x}}\xrightarrow{M_{\bar{x}}(t)} M_c\label{eq:0}\\
           \eta_{t^{\bar{x}}(t_1,\ldots, t_n)} &\colon N(t^{\bar{x}}(t_1,\ldots,t_n))F_0\xRightarrow{ N_{\bar{x}}(t)(\eta_{t_1},\ldots, \eta_{t_n})}N_{\bar{x}}(t)F_{\bar{x}}(M(t_1),\ldots, M(t_n))
           \nonumber\\
           &\quad \xRightarrow{\eta_{t,\bar{x}}(M(t_1),\ldots, M(t_n))}F_cM(t^{\bar{x}}(t_1,\ldots, t_n))\label{eq:1}
   \end{align}
\end{lemma}
\begin{proof}
We prove the claim by induction on the structure of the term $t$.

\noindent\textbf{Case 1:} Assume $t = x_i$. Then $t^{\bar{x}}(t_1,\ldots, t_n) = t_i$, $M(t_i) = \pi_i(M(t_1),\ldots, M(t_n))$ and $\eta_{t_i} = \id\bullet \pi_i(\eta_{t_1},\ldots, \eta_{t_n})$. Thus the equations hold.

\noindent\textbf{Case 2:} Assume $t = B(t_1',\ldots, t_k')$, where the equations hold for $t_i'$ for $i\leq k$ and $B$ is a functor symbol in $\Sigma$.
The substitution on $t$ then yields:
\[
t^{\bar{x}}(t_1, \dots, t_n) = B(t_1'^{\bar{x}}(t_1,\ldots, t_n),\ldots, t_k'^{\bar{x}}(t_1,\ldots, t_n))= B(s_1, \dots, s_k),
\]
where $s_i = t_i'^{\bar{x}}(t_1,\ldots, t_n)$ for $i\leq k$. By the induction hypothesis $M(s_i) = M_{\bar{x}}(t_i')\bar{M}$, where $\bar{M}\coloneqq (M(t_1),\ldots, M(t_n))$ for $i\leq k$.
To verify Equation~\eqref{eq:0}, consider:
\begin{align*}
   &M(t^{\bar{x}}(t_1, \dots, t_n)) \\
   &= M(B(s_1, \dots, s_k)) \\
   &= M(B)(M(s_1), \dots, M(s_k)) \tag{Definition of $M(B(s_1,\dots, s_k))$} \\
   &= M(B) (M_{\bar{x}}(t'_1)\bar{M}, \dots, M_{\bar{x}}(t'_k)\bar{M}) \tag{Induction hypothesis \eqref{eq:0} on $t'_i$} \\
   &= M(B) (M_{\bar{x}}(t'_1), \dots, M_{\bar{x}}(t'_k)) \bar{M} \\
   &= M_{\bar{x}}(B(t'_1, \dots, t'_k))\bar{M} \\
   &= M_{\bar{x}}(t)(M(t_1), \dots, M(t_n))
\end{align*}
This proves Equation~\eqref{eq:0} in the inductive step. Denote $\bar{\eta}\coloneqq (\eta_{t_1},\ldots, \eta_{t_n})$. For Equation~\eqref{eq:1}, consider:
{
\begin{align*}
   &\eta_{t^{\bar{x}}(t_1,\ldots, t_n)}\\
   &= \eta_{B(s_1,\ldots, s_k)}\\
   &= \eta^B (M(s_1),\ldots, M(s_k)) \bullet N(B)(\eta_{s_1},\ldots, \eta_{s_k}) \tag{Definition of $\eta_{B(s_1,\ldots, s_k)}$}\\
   &= \eta^B(M_{\bar{x}}(t_1')\bar{M},\ldots, M_{\bar{x}}(t_k')\bar{M})\bullet\\
   & N(B)(\eta_{t_1',\bar{x}}\bar{M} \bullet N_{\bar{x}}(t_1')\bar{\eta},\ldots, \eta_{t_k',\bar{x}}\bar{M}\bullet N_{\bar{x}}(t_k')\bar{\eta}) \tag{Induction hypothesis on $t_i'$}\\
   &=\eta^B(M_{\bar{x}}(t_1')\bar{M},\ldots, M_{\bar{x}}(t_k')\bar{M})\bullet
   N(B)(\eta_{t_1',\bar{x}}\bar{M},\ldots, \eta_{t_k',\bar{x}}\bar{M})\bullet\\
   & N(B)(N_{\bar{x}}(t_1')\bar{\eta},\ldots, N_{\bar{x}}(t_k')\bar{\eta}) \tag{Composition componentwise}\\
   &= (\eta^B(M_{\bar{x}}(t_1'),\ldots, M_{\bar{x}}(t_k'))\bullet
   N(B)(\eta_{t_1',\bar{x}},\ldots, \eta_{t_k',\bar{x}}))\bar{M} \bullet
    N(B)(N_{\bar{x}}(t_1'),\ldots, N_{\bar{x}}(t_k'))\bar{\eta} \\
   &= \eta_{B(t_1',\ldots, t_k'),\bar{x}}\bar{M}\bullet N_{\bar{x}}(B(t_1',\ldots, t_k'))\bar{\eta}\tag{Definitions of $\eta_{B(t_1',\ldots, t_k')}$ and $N_{\bar{x}}(B(t_1',\ldots, t_k'))$}\\
   &=\eta_{t,\bar{x}}(M(t_1),\ldots, M(t_n))\bullet N_{\bar{x}}(t)(\eta_{t_1},\ldots, \eta_{t_n})\tag{$t = B(t_1',\ldots, t_k')$}
\end{align*}
}
\end{proof}

\subsection{Examples}

\begin{examples}[Examples of $2$-Signatures]\label{ex:2-signatures}
   We define the following standard $2$-signatures representing background categorical structures.
   {\small
   \begin{align*}
       \Sigma_{\text{Cart}}
    &= (c,\quad \square\colon cc\to c,\ I\colon ()\to c,
    \quad x\xrightarrow{\delta^x}x\square x, \ x\xleftarrow{p_1^{xy}}x\square y\xrightarrow{p_2^{xy}}y,\ x\xrightarrow{!^x}I)\\
       \Sigma_{\text{Mon}}
       &= (c,\quad \square\colon cc\to c,\  I\colon ()\to c,\quad (x\square y)\square z\rightleftarrows x\square (y\square z),\ x\square I\rightleftarrows x\leftrightarrows I\square x)\\
       \Sigma_{\text{SMon}}
       &= (c,\quad \square\colon cc\to c,\ I\colon ()\to c,\quad (x\square y)\square z\rightleftarrows x\square (y\square z),\  x\square I\rightleftarrows x\leftrightarrows I\square x,\ x\square y\to y\square x)\\
       \Sigma_{\text{Adj}}
       &= (c,d,\quad L\colon c\to d,\ R\colon d\to c,\quad \eta^x\colon x\to R(L(x)),\ \varepsilon^y\colon L(R(y))\to y)\\
       \Sigma_{\text{Mnd}}
       &= (c,\quad T\colon c\to c,\quad \mu^x\colon T(T(x))\to T(x),\ \eta^x\colon x\to T(x))\\
       \Sigma_{\text{Distr}}
       &= (c,\quad T,S\colon c\to c,\quad  \mu_T^x,\ \eta_T^x,\ \mu_S^x,\ \eta_S^x,\ \gamma^x\colon S(T(x))\to T(S(x)))
   \end{align*}
   }
\end{examples}

\begin{examples}\label{exs:theories}
   We give the theories of monoids, Hopf algebras and fields, and a family of theories generalizing lax limits:
   \begin{enumerate}
       \item \textbf{Monoids:} Consider the signature pair $(\Sigma, \sigma)$ defined as follows:
       \begin{align*}
       \Sigma
       &= (c,\quad \square\colon cc\to c, I\colon ()\to c,\quad (x\square y)\square z\xrightarrow{\alpha^{xyz}} x\square (y\square z), x\square I\xrightarrow{\rho^x} x\xleftarrow{\lambda^x} I\square x)\\
       \sigma
       &= (s\colon c,\quad \mu\colon s\square s\to s, u\colon I\to s)
   \end{align*}
   For convenience, we write $x\square y$ for $\square(x,y)$. Consider the following parallel paths in the term graph $G_{\sigma}$ defining the theory of monoids $T$:
   \[
   \begin{tikzcd}[column sep=0.8 cm]
   (s \square s) \square s \arrow[d, "\mu \square s"'] \arrow[r, "\alpha_{s,s,s}"] & s \square (s \square s) \arrow[r, "s \square \mu"] & s \square s \arrow[d, "\mu"] & s \square I \arrow[r, "s \square u"] \arrow[rd, "\rho_s"'] & s \square s \arrow[d, "\mu"] & s \square s \arrow[d, "\mu"'] & I \square s \arrow[l, "u \square s"'] \arrow[ld, "\lambda_s"] \\
   s \square s \arrow[rr, "\mu"'] & & s & & s & s &
   \end{tikzcd}
   \]
   Let $M$ be a $\Sigma$-model in $\CAT$, for instance a monoidal category. We write $\otimes\coloneqq M(\square)$ and $I$ for $M(I)(*)$, where $*$ is the unique object of the terminal category. Note that
   \[
   M^\sigma = \Ins((M(s\square s),M(I)), (M(s),M(s))),
   \]
   where $M(s\square s)$, $M(s)$ and $M(I)$ are the functors
   \[
   x\mapsto x\otimes x, \quad x, \quad I\colon M_c\to M_c\text{, respectively.}
   \]
   Thus $M^\sigma$ consists of objects $(x,m\colon x\otimes x\to x, e\colon I\to x)$, and $M^\sigma$ can be understood as the category of pointed magmas of $M$. Let $U_M\colon M^\sigma\to M_c$ denote the forgetful functor. The internal monoids in $M$ form the full subcategory $M^T$ of $M^\sigma$.

   By definition, the category $M^T$ is the equifier of the parallel pair of natural transformations:
   \begin{align*}
   &\left(\overline{\mu \circ (\mu \square s)}_M, \,\, \overline{\rho_s}_M, \,\, \overline{\lambda_s}_M\right), \,\, \left(\overline{\mu \circ (s \square \mu) \circ \alpha_{s,s,s}}_M, \,\, \overline{\mu \circ (s \square u)}_M, \,\, \overline{\mu \circ (u \square s)}_M\right)\colon \\
   &\left(M((s \square s)\square s)U_M, \,\, M(s \square I)U_M, \,\, M(I \square s)U_M\right) \Rightarrow (U_M, \,\, U_M, \,\, U_M)\colon M^\sigma\to M_c^3
   \end{align*}
   This means an object $(x, m, e)$ of $M^\sigma$ is in $M^T$ if and only if
   {\small
   \[
   \left(\overline{\mu \circ (\mu \square s)}_M, \,\, \overline{\rho_s}_M, \,\, \overline{\lambda_s}_M\right)_{(x,m,e)} = \left(\overline{\mu \circ (s \square \mu) \circ \alpha_{s,s,s}}_M, \,\, \overline{\mu \circ (s \square u)}_M, \,\, \overline{\mu \circ (u \square s)}_M\right)_{(x,m,e)},
   \]
   that is, if and only if the following three equations hold in $M_c$:
   \begin{align*}
       m \circ (m \otimes x) &= m \circ (x \otimes m) \circ M(\alpha)_{x,x,x} \\
       m \circ (x \otimes e) &= M(\rho)_x \\
       m \circ (e \otimes x) &= M(\lambda)_x
   \end{align*}
   }
   These are precisely the classical associativity and unit laws for a monoid in $M$.

\item \textbf{Hopf Algebras:}
    To define internal Hopf algebras, we work relative to a skew-monoidal $2$-signature with a chosen interchange law:
    {\small
    \begin{align*}
        \Sigma&=(c, \quad \square \colon cc \to c,\ I \colon () \to c, \\
        &\quad (x \square y) \square z\xrightarrow{\alpha^{xyz}}  x \square (y \square z),\ I \square x \xrightarrow{\lambda^x} x \xrightarrow{\rho^x} x \square I,\ (x \square y) \square (z \square v) \xrightarrow{\chi^{xyzv}} (x \square z) \square (y \square v))
    \end{align*}
    }
    where $\chi^{xyzv}$ represents the interchange law, which is often uniquely constructed via the coherence of symmetric, or more generally, braided monoidal categories \cite{MacLane1998, joyalBraidedTensorCategories1993, streetQuantumGroupsPath2007}.

    We define the relative $1$-signature:
    \[
    \sigma_{\text{Hopf}} \coloneqq (s\colon c, \quad \mu\colon s \square s \to s, \ u\colon I \to s, \ \delta\colon s \to s \square s, \ \varepsilon\colon s \to I, \ S\colon s \to s)
    \]
    The theory $T_{\text{Hopf}}$ consists of the following equations in the term graph $G_{\sigma_{\text{Hopf}}}$, presented as commutative diagrams where each face represents a parallel pair of paths:\footnote{We use $p_1\square p_2$ for the path $(p_1\square t_2')\circ (t_1\square p_2)$, where $p_i\colon t_i\to t_i'$ is an arrow in the term graph $G_{\sigma_{\mathrm{Hopf}}}$ for $i = 1,2$.}
    {\small \[
\begin{tikzcd}[column sep = 0.7 cm]
(s\square s)\square s \arrow[d, "\mu\square s"'] \arrow[r, "{\alpha_{s,s,s}}"] & s\square (s\square s) \arrow[r, "s\square \mu"]              & s\square s \arrow[d, "\mu"]                                         & I\square s \arrow[rd, "\lambda_s"'] \arrow[r, "u\square s"]              & s\square s \arrow[d, "\mu" description]                                         & s\square I \arrow[l, "s\square u"']    &                        \\
s\square s \arrow[rr, "\mu"']                                                  &                                                              & s                                                                   &                                                                          & s                                                                               & s \arrow[l, "="'] \arrow[u, "\rho_s"'] &                        \\
s \arrow[d, "\delta"'] \arrow[r, "\delta"]                                     & s\square s \arrow[r, "\delta\square s"]                      & (s\square s)\square s \arrow[d, "{\alpha_{s,s,s}}"]                 & I\square s \arrow[d, "\lambda_s"']                                       & s\square s \arrow[l, "\varepsilon\square s"'] \arrow[r, "s\square \varepsilon"] & s\square I                             &                        \\
s\square s \arrow[rr, "s\square \delta"']                                      &                                                              & s\square(s\square s)                                                & s                                                                        & s \arrow[l, "="'] \arrow[u, "\delta" description] \arrow[ru, "\rho_s"']         &                                        &                        \\
(s\square s)\square (s\square s) \arrow[r, "{\chi_{s,s,s,s}}"]                 & (s\square s)\square (s\square s) \arrow[r, "\mu\square \mu"] & s\square s                                                          & I\square I \arrow[r, "\lambda_I"]                                        & I                                                                               & I\square I \arrow[r, "u\square u"]     & s\square s             \\
s\square s \arrow[u, "\delta\square \delta"] \arrow[rr, "\mu"']                &                                                              & s \arrow[u, "\delta"']                                              & s\square s \arrow[r, "\mu"'] \arrow[u, "\varepsilon\square \varepsilon"] & s \arrow[u, "\varepsilon"']                                                     & I \arrow[r, "u"'] \arrow[u, "\rho_I"]  & s \arrow[u, "\delta"'] \\
                                                                               & I                                                            & s\square s \arrow[rr, "s\square S"]                                 &                                                                          & s\square s \arrow[d, "\mu"]                                                     &                                        &                        \\
I \arrow[r, "u"'] \arrow[ru, "="]                                              & s \arrow[u, "\varepsilon"']                                  & s \arrow[r, "\varepsilon"] \arrow[u, "\delta"] \arrow[d, "\delta"'] & I \arrow[r, "u"]                                                         & s                                                                               &                                        &                        \\
                                                                               &                                                              & s\square s \arrow[rr, "S\square s"']                                &                                                                          & s\square s \arrow[u, "\mu"']                                                    &                                        &
\end{tikzcd}
\]
}

    Evaluating this theory in $\CAT$ over a symmetric monoidal category $M$ yields precisely the classical category $\Hopf_M$ of internal Hopf algebras.

\item \textbf{Fields:}
    The category of fields can be characterized starting from the category of commutative groups. Specifically, a field $k$ corresponds to a commutative group $k^*$ with a commutative ring structure on $k^*+1$ whose multiplication extends that of $k^*$.

    To formalize this, we consider a signature pair $(\Sigma, \sigma)$:
    \begin{align*}
    \Sigma
    &\coloneqq (c,\quad \square\colon cc\to c,\ I\colon ()\to c,\ S\colon c\to c,\\
    &\quad\quad x\xrightarrow{\delta^x}x\square x, \ x\xleftarrow{p_1^{xy}}x\square y\xrightarrow{p_2^{xy}}y,\ x\xrightarrow{!^x}I,\ x\xrightarrow{\eta^x}Sx)\\
    \sigma
    &\coloneqq (s \colon c,\quad m\colon s\square s\to s,\ i\colon s\to s,\ u\colon I\to s,\\
    &\quad\quad m_1,m_2\colon Ss\square Ss\to Ss,\ \nu\colon Ss\to Ss,\ \iota\colon I\to Ss)
    \end{align*}
    To define the bialgebraic theory of fields, we introduce the theory $T$ containing precisely the following equations in the term graph $G_\sigma$:
    \begin{itemize}
        \item The equations stating that $(s, m, i, u)$ satisfies the axioms of a commutative group.
        \item The equations stating that $(Ss, m_1, \nu, \iota, m_2, \eta_s \circ u)$ determines a commutative ring, where $(Ss, m_1, \nu, \iota)$ and $(Ss, m_2, \eta_s \circ u)$ satisfy the additive and the multiplicative equations, respectively.
        \item An equation stating that $\eta_s \colon s \to Ss$ determines a morphism of magmas from $(s,m)$ to $(Ss,m_2)$:
        \[
        m_2 \circ (\eta_s \square \eta_s) = \eta_s \circ m \colon s \square s \to Ss
        \]
    \end{itemize}
    We choose the categorical $\Sigma$-model $M = (\Set, \times, 1, (-)+1)$ where $M(\eta)$ is the unit of the maybe monad $(-)+1$. Then $M^T$ is equivalent to the category of fields.\footnote{One could take a bifunctor symbol $\boxplus$ with the codiagonal and injections instead of $S, \eta$ and $\iota$. This would be more advantageous in the sense that the coproduct can be defined as the left adjoint to the diagonal $1$-cell. This would allow an essentially unique formulation for fields if we were to develop the notion of $2$-equational $\Sigma$-theories.}

\item \textbf{Generalizing Lax Limits:}
    Consider an arbitrary $2$-signature $\Sigma = (C,\mathcal{F},\mathcal{T})$. We construct an associated algebraic theory $T$ for $\Sigma$ which generalizes the notion of a lax limit of a strict $2$-functor \cite{streetLimitsIndexedCategoryvalued1976, lack2CategoriesCompanion2010} to the cartesian (multivariable) setting. Specifically, this construction unifies the theories of monoids, commutative monoids, and monad algebras.

    We define the $1$-signature $\sigma_{\text{Lax}}$ relative to $\Sigma$:
    \[
    \sigma_{\text{Lax}} \coloneqq \left(s_c\colon c, \quad f_B\colon B(s_{c_1},\dots, s_{c_n})\to s_d\right)_{c\in C,B\colon c_1\cdots c_n\to d\in\mathcal{F}}.
    \]
    That is, the sorts are indexed by $C$ and the function symbols by $\mathcal{F}$. We recursively define for each term $t \colon c \in \Term_S^\Sigma$ a path $f_t \colon t \to s_c$ in the term graph $G_{\sigma_{\text{Lax}}}$:
    {\small
    \[
    f_t\coloneqq \begin{cases}
        ()_{s}\colon s\to s, & \text{if $t = s\in S$,}\\
        B(t_1, \dots, t_n) \xrightarrow{B(f_{t_1}, \dots, f_{t_n})} B(s_{c_1}, \dots, s_{c_n}) \xrightarrow{f_B} s_c, & \text{if $t = B(t_1,\dots, t_n)$, $t_i\colon c_i$ for $i\leq n$,}
    \end{cases}
    \]
    }
    where we use the convention that for any paths $p^i = (p^i_{m_i},\dots, p^i_1)\colon t_i'\to t_i''$ in $G_\sigma$ (for $i\leq n$) and a $\Sigma$-functor symbol $B$, the term $B(p^1, \dots, p^n)$ denotes the composition:
    \begin{align*}
        B(p^1,\dots, p^n) &\coloneqq B(p^1_{m_1},t_2'',\dots, t_n'')\circ\dots\circ B(p^1_1,t_2'',\dots, t_n'') \circ \cdots \\
        &\quad \dots \circ B(t_1',\dots, t_{n-1}',p^n_{m_n})\circ\dots\circ B(t_1',\dots, t_{n-1}', p^n_1).
    \end{align*}
    The theory $T_{\text{Lax}}$ consists of the following parallel paths:
    \[
    \begin{tikzcd}
    t_1^{\bar{x}}(\bar{s}) \arrow[r, "\alpha_{\bar{s}}"] \arrow[rd, "f_{t_1^{\bar{x}}(\bar{s})}"'] & t_2^{\bar{x}}(\bar{s}) \arrow[d, "f_{t_2^{\bar{x}}(\bar{s})}"] \\
                                                                                                   & s_c
    \end{tikzcd}
    \]
    for each $\alpha^{\bar{x}}\colon t_1\to t_2\in \mathcal{T}$, where $\bar{s} = s_{c_1}\cdots s_{c_n}$ and $x_i\colon c_i$ for $i\leq n$.

    If $\Sigma$ is the signature for a monad, a monoidal category or a symmetric monoidal category, then the associated algebraic theory is the theory of monad algebras, monoids and commutative monoids, respectively.
       \end{enumerate}

\end{examples}
\subsection{Induced Functor of Algebras}

In this subsection, we fix a PIE-complete $2$-category $\K$, a signature pair $(\Sigma,\sigma)$ and a bialgebraic $\sigma$-theory $T$. We find conditions under which a morphism $F\colon M\to N$ of $\Sigma$-models in $\K$ induces a morphism $F^T\colon M^T\to N^T$ between the objects of $T$-bialgebras.

\begin{lemma}\label{lem:morphism_between_inserters}
   Let $(F,\eta) \colon M \to N$ be a lax morphism of $\Sigma$-models in $\K$. Assume that $\eta_t\colon N(t)F_0\Rightarrow F_c M(t)\colon M_0\to N_c$ is an isomorphism for each function coarity term $t\colon c$. Then the following hold:
   \begin{enumerate}
       \item There exists a unique morphism $F^\sigma\colon M^\sigma\to N^\sigma$ satisfying the equations
       \begin{align}
           u_N \circ F^\sigma
           &= F_0 \circ u_M \nonumber \\
           \phi_N * F^\sigma
           &= (\eta_{t_2}^{-1}*u_M)_f \bullet ((\prod_{f\colon t_1\to t_2\colon c} F_{c}) * \phi_M) \bullet (\eta_{t_1}*u_M)_f\label{eq:2}
   \end{align}
       \item For every path $p\colon t_1\to t_2$ in the term graph $G_\sigma$, the following diagram commutes:
       \[
\begin{tikzcd}
N(t_1)F_0u_M \arrow[r, "\eta_{t_1}*u_M", Rightarrow] \arrow[d, "\bar{p}_N*F^\sigma"', Rightarrow] & F_cM(t_1)u_M \arrow[d, "F_c*\bar{p}_M", Rightarrow] \\
N(t_2)F_0u_M \arrow[r, "\eta_{t_2}*u_M"', Rightarrow]                                             & F_cM(t_2)u_M
\end{tikzcd}
\]
In other words,
\begin{equation}\label{eq:3}
\eta_{t_2}u_M\bullet \bar{p}_NF^\sigma = F_c\bar{p}_M\bullet \eta_{t_1}u_M
\end{equation}

   \end{enumerate}
\end{lemma}
\begin{proof}
   The first universal property of the inserter $N^\sigma$ induces the unique morphism $F^\sigma$:
   \[
   \begin{tikzcd}[row sep = 1 cm]
M^\sigma \arrow[d, "u_M"'] \arrow[r, "u_M"]                                                 & M_0 \arrow[r, "F_0"] \arrow[d, "(M(t_2))_f" description]                                                                              & N_0 \arrow[dd, "(N(t_2))_f"] \\
M_0 \arrow[d, "F_0"'] \arrow[r, "(M(t_1))_f"'] \arrow[ru, "\phi_M" description, Rightarrow] & \prod_{f\colon t_1\to t_2\colon c}M_c \arrow[rd, "\prod_f F_c" description] \arrow[ru, "(\eta_{t_2}^{-1})_f" description, Rightarrow] &                              \\
N_0 \arrow[rr, "(N(t_1))_f"'] \arrow[ru, "(\eta_{t_1})_f" description, Rightarrow]          &                                                                                                                                       & \prod_f N_c
\end{tikzcd}
   \]
   We prove by structural induction on arrows $p$ in $G_\sigma$ that Equation~\eqref{eq:3} holds:

       \noindent\textbf{Case 1:} Assume $p = f$ is a $\sigma$-function symbol. Whiskering Equation~\eqref{eq:2} with $\pi_f$ from the left yields Equation~\eqref{eq:3}.

   \noindent\textbf{Case 2:} Assume $p = \alpha_{t_1',\ldots, t_n'}$, where $\alpha^{\bar{x}}\colon t_1''\to t_2''$ is a transformation symbol in $\Sigma$ and $t_i = t_i''^{\bar{x}}(t_1',\ldots, t_n')$ for $i = 1,2$.
   Now
     {\small \begin{align*}
       &\eta_{t_2}u_M\bullet \bar{p}_NF^\sigma\\
       &= \eta_{t_2}u_M\bullet N(\alpha)(N^\sigma(t_1'),\ldots, N^\sigma(t_n'))F^\sigma\tag{Definition of $\bar{p}_N $}\\
       &= \eta_{t_2}u_M\bullet N(\alpha)(N(t_1')F_0u_M, \ldots, N(t_n')F_0u_M)\tag{$N^\sigma(t_i')F^\sigma = N(t_i')u_N F^\sigma = N(t_i')F_0u_M$}\\
       &= (\eta_{t_2}\bullet N(\alpha)(N(t_1')F_0,\dots, N(t_n')F_0))u_M \tag{Whiskering}\\
       &= (\eta_{t_2'',\bar{x}}(M(t_1'),\ldots, M(t_n'))\bullet N_{\bar{x}}(t_2'')(\eta_{t_1'},\ldots, \eta_{t_n'})\bullet N(\alpha)(N(t_1')F_0,\dots, N(t_n')F_0))u_M \tag{\ref{eq:1}}\\
       &= (\eta_{t_2'',\bar{x}}(M(t_1'),\ldots, M(t_n'))\bullet N(\alpha)F_{\bar{x}}(M(t_1'),\dots, M(t_n'))\bullet N_{\bar{x}}(t_1'')(\eta_{t_1'},\ldots, \eta_{t_n'}))u_M\tag{Interchange}\\
       &= ((\eta_{t_2'',\bar{x}}\bullet N(\alpha)F_{\bar{x}})(M(t_1'),\ldots, M(t_n'))\bullet N_{\bar{x}}(t_1'')(\eta_{t_1'},\ldots, \eta_{t_n'}))u_M\tag{Whiskering}\\
       &= ((F_cM(\alpha)\bullet\eta_{t_1'',\bar{x}})(M(t_1'),\ldots, M(t_n'))\bullet N_{\bar{x}}(t_1'')(\eta_{t_1'},\ldots, \eta_{t_n'}))u_M\tag{Definition of lax morphism}\\
       &= (F_cM(\alpha)(M(t_1'),\ldots, M(t_n'))\bullet \eta_{t_1'',\bar{x}}(M(t_1'),\ldots, M(t_n'))\bullet N_{\bar{x}}(t_1'')(\eta_{t_1'},\ldots, \eta_{t_n'}))u_M \tag{Whiskering}\\
       &= (F_cM(\alpha)(M(t_1'),\ldots, M(t_n'))\bullet \eta_{t_1})u_M\tag{\ref{eq:1}}\\
       &= F_c\bar{p}_M \bullet \eta_{t_1}u_M\tag{Whiskering + Definition of $\bar{p}_M$}
   \end{align*}
   }

  \noindent \textbf{Case 3:} Assume $p = B(t_1',\ldots, q,\ldots, t_n')\colon B(t_1',\ldots, t_i',\ldots, t_n')\to B(t_1',\ldots, t_i'',\ldots, t_n')$, where $B\colon \bar{c} = c_1\cdots c_n\to c$ is a functor symbol in $\Sigma$ and the claim holds for the arrow $q\colon t_i'\to t_i''$. By definition, $\eta_{t_1}$ is the composite
  {\small\[
  N(B)(N(t_1'),\ldots, N(t_n'))F_0\xRightarrow{N(B)(\eta_{t_1'},\ldots, \eta_{t_n'})} N(B)F_{\bar{c}}(M(t_1'),\ldots, M(t_n'))\xRightarrow{\eta^B(M(t_1'),\ldots, M(t_n'))} F_cM(t_1)
  \]}
  Denote $M'' = (M^\sigma(t_1'),\ldots, M^\sigma(t_i''),\ldots,  M^\sigma(t_n'))$ and $M' = (M^\sigma(t_1'),\ldots, M^\sigma(t_n'))$. Now
  {\small\noindent
  \begin{align*}
       &\eta_{t_2}u_M\bullet \bar{p}_NF^\sigma\\
       &= \eta^B M''\bullet N(B)(\eta_{t_1'}u_M,\ldots,\eta_{t_i''}u_M,\ldots, \eta_{t_n'}u_M)\bullet\\
       & N(B)(N^\sigma(t_1')F^\sigma,\ldots,\bar{q}_N F^\sigma,\ldots, N^\sigma(t_n')F^\sigma) \tag{Definitions of $\eta_{t_2}$ and $\bar{p}_N$}\\
       &=  \eta^BM''\bullet N(B)(\eta_{t_1'}u_M, \ldots, \eta_{t_i''}u_M\bullet\bar{q}_NF^\sigma, \ldots, \eta_{t_n'}u_M)\\
       &= \eta^B M''\bullet N(B)(\eta_{t_1'}u_M,\ldots, F_{c_i}\bar{q}_M\bullet \eta_{t_i'}u_M,\ldots, \eta_{t_n'}u_M)\tag{Induction on $q$}\\
       &= \eta^B M''\bullet N(B)F_{\bar{c}}(M^\sigma(t_1'),\ldots,\bar{q}_M,\ldots, M^\sigma(t_n'))\bullet N(B)(\eta_{t_1'}u_M,\ldots, \eta_{t_n'}u_M)\\
       &= F_cM(B)(M^\sigma(t_1'),\ldots, \bar{q}_M,\ldots, M^\sigma(t_n'))\bullet \eta^B M'\bullet
       N(B)(\eta_{t_1'}u_M,\ldots, \eta_{t_n'}u_M)\tag{Interchange}\\
       &= F_c\bar{p}_M\bullet \eta_{t_1}u_M \tag{Definitions of $\bar{p}_M$ and $\eta_{t_1}$}
  \end{align*}
  }

  Thus Equation~\eqref{eq:3} holds for all arrows in $G_\sigma$, and it remains to prove it for paths. For the empty path it holds trivially. Let $p = p_n\cdots p_1$ be a path of arrows in $G_\sigma$. We treat the case $n = 2$; the general case follows by induction on $n$. Write $p\colon t_1'\xrightarrow{p_1}t_2'\xrightarrow{p_2}t_3'$. Then
  \begin{align*}
       &\eta_{t_3'}u_M\bullet \bar{p}_NF^\sigma\\
       &= \eta_{t_3'}u_M\bullet \bar{p_2}_NF^\sigma \bullet \bar{p_1}_NF^\sigma \tag{$p = p_2p_1$ + Whiskering}\\
       &= F_c\bar{p_2}_M\bullet \eta_{t_2'}u_M\bullet \bar{p_1}_NF^\sigma \tag{Equation~\eqref{eq:3} holds for $p_2$}\\
       &= F_c\bar{p_2}_M\bullet F_c\bar{p_1}_M\bullet \eta_{t_1'}u_M \tag{Equation~\eqref{eq:3} holds for $p_1$}\\
       &= F_c\bar{p}_M\bullet \eta_{t_1'}u_M \tag{$p = p_2p_1$ + Whiskering}
  \end{align*}\qedhere

\end{proof}
\begin{theorem}[Induced Functor of Algebras Theorem]\label{thm:induced_morphism_between_algebras}
   Let $(F,\eta)\colon M\to N$ be a lax morphism of $\Sigma$-models in $\K$. Assume the following:
\begin{enumerate}
   \item $\eta_t\colon N(t)F_0\Rightarrow F_cM(t)$ is an isomorphism for all function coarity terms $t\colon c$.
   \item $\eta_t*u_Me_M\colon N(t)F_0u_Me_M\Rightarrow F_cM(t)u_Me_M$ is a monomorphism in $\K(M^T,N_c)$ for every equational coarity term $t\colon c$.
\end{enumerate}
Then there is a unique morphism $F^T\colon M^T\to N^T$ such that $e_N F^T = F^\sigma e_M$.
\end{theorem}
\begin{proof}
   The first assumption with Lemma~\ref{lem:morphism_between_inserters}(1) yields the unique morphism $F^\sigma\colon M^\sigma\to N^\sigma$. Consider the following diagram, where the indexing is over $(p,q\colon t_1\to t_2\colon c)\in T$:
   \[
\begin{tikzcd}[ampersand replacement=\&]
 {M^T} \& {M^\sigma} \&\& {\prod_{(p,q)} M_c} \\
 \\
 {N^T} \& {N^\sigma} \&\& {\prod_{(p,q)} N_c}
 \arrow["{{{{e_M}}}}", from=1-1, to=1-2]
 \arrow["{{\exists! F^T}}"{description}, dashed, from=1-1, to=3-1]
 \arrow[""{name=0, anchor=center, inner sep=0}, "{{{{(M^\sigma(t_1))}}}}", shift left=3, from=1-2, to=1-4]
 \arrow[""{name=1, anchor=center, inner sep=0}, "{{{{(M^\sigma(t_2))}}}}"', shift right=3, from=1-2, to=1-4]
 \arrow["{{{F^\sigma}}}"', from=1-2, to=3-2]
 \arrow["{{{\prod_{(p,q)}F_c}}}", from=1-4, to=3-4]
 \arrow["{{e_N}}"', from=3-1, to=3-2]
 \arrow[""{name=2, anchor=center, inner sep=0}, "{{{{(N^\sigma(t_1))}}}}", shift left=3, from=3-2, to=3-4]
 \arrow[""{name=3, anchor=center, inner sep=0}, "{{{{(N^\sigma(t_2))}}}}"', shift right=3, from=3-2, to=3-4]
 \arrow["{{{{(\bar{p}_M)}}}}"', shift right=3, Rightarrow, from=0, to=1]
 \arrow["{{{{(\bar{q}_M)}}}}", shift left=3, Rightarrow, from=0, to=1]
 \arrow["{{{{(\bar{q}_N)}}}}", shift left=3, Rightarrow, from=2, to=3]
 \arrow["{{{{(\bar{p}_N)}}}}"', shift right=3, Rightarrow, from=2, to=3]
\end{tikzcd}
\]
   To induce the morphism $F^T$ using the equifier, we must show that $\bar{p}_N*F^\sigma e_M = \bar{q}_N*F^\sigma e_M$ for every $(p,q)\in T$. Let $(p,q\colon t_1\to t_2\colon c)\in T$. Then
   \begin{align*}
       \eta_{t_2}u_Me_M\bullet \bar{p}_NF^\sigma e_M
       &= F_c\bar{p}_Me_M\bullet \eta_{t_1}u_Me_M \tag{\eqref{eq:3}}\\
       &= F_c\bar{q}_Me_M\bullet \eta_{t_1}u_Me_M \tag{Equifier property of $e_M$}\\
       &= \eta_{t_2}u_Me_M\bullet \bar{q}_NF^\sigma e_M. \tag{\eqref{eq:3}}
   \end{align*}
   By the second assumption $\eta_{t_2}*u_Me_M$ is monic and hence $\bar{p}_N*F^\sigma e_M = \bar{q}_N*F^\sigma e_M$. Thus the equifier $N^T$ induces the unique morphism $F^T\colon M^T\to N^T$ with $e_NF^T = F^\sigma e_M$.
\end{proof}
\begin{corollary}
   Let $(F,\eta) \colon M \to N$ be a morphism of $\Sigma$-models in $\K$.
   \begin{enumerate}
       \item If $F$ is lax and $T$ is algebraic, then $F$ induces a $1$-cell $F^T \colon M^T \to N^T$ in $\K$.
       \item If $F$ is oplax and $T$ is coalgebraic, then $F$ induces a $1$-cell $F^T \colon M^T \to N^T$ in $\K$.
       \item If $F$ is strong, then $F$ induces a $1$-cell $F^T \colon M^T \to N^T$ in $\K$.
   \end{enumerate}
\end{corollary}
\begin{proof}
The lax case follows from Theorem~\ref{thm:induced_morphism_between_algebras}, since if $T$ is algebraic, then $\eta_t$ is an identity for each coarity term $t$ as coarity terms are sorts. The oplax case follows similarly. The strong case follows from Theorem~\ref{thm:induced_morphism_between_algebras}, since we inductively see that $\eta_t$ is an isomorphism for each term $t$.
\end{proof}

As the theory of monoids is algebraic, we recover the following classical result:
\begin{corollary}
   Let $F\colon M\to N$ be a lax monoidal functor between monoidal categories. Then there is an induced functor $\bar{F}\colon \Mon(M)\to \Mon(N)$ between the categories of internal monoids.
\end{corollary}

\section{Lifting Limits and Colimits}\label{sec:lifting_limits}
In this section, we study limits within an object $c$ of a $2$-category $\K$. A natural notion for a limit in $c$ is a right Kan extension of a morphism $d\colon I\to c$ along a morphism $k\colon I\to K$. Ordinary limits are obtained when $\K$ is the $2$-category $\CAT$ and $K$ is the terminal category. Weighted limits in the enriched setting can also be expressed as Kan extensions \cite{Kelly:Enriched}.

Understanding the lifting of Kan extensions along PIE limits allows us to prove completeness and monadicity properties of the objects $\BialgB$ of bialgebras.

\subsection{Kan Extensions and Their Strict Creation}

\begin{definition}
   Let $\C$ be a category. We call an object $c$ in $\C$ a \textbf{quotient initial} object if for any object $c'$ in $\C$ there exists at most one morphism $c\to c'$. \textbf{Subterminality} of an object is defined dually.
\end{definition}
Quotient initiality of a cocone $\lambda = (\lambda_i)_i\colon D\Rightarrow c\colon I\to \C$, where $c$ refers to the constant functor at an object $c$ in $\C$, means that $\lambda$ is jointly epic as a family of morphisms to $c$. Similarly, subterminality of a cone over a diagram is the same as the cone determining a jointly monic family of morphisms.
\begin{definition}[Kan Extensions]
   Let $K\xleftarrow{k}I\xrightarrow{d}c$ be a span of morphisms in a $2$-category $\K$. Consider the pre-composition functor $k^*\colon \K(K,c)\to \K(I,c)$. We define the following notions:
   \begin{enumerate}
       \item We call the comma categories $\LE(d,k) \coloneqq d\downarrow k^*$ and $\RE(d,k)\coloneqq k^*\downarrow d$ the categories of \textbf{left} and \textbf{right extensions} of $d$ \textbf{along} $k$, respectively.
       \item An initial (terminal) left (right) extension of $d$ along $k$ is called a \textbf{left} (\textbf{right}) \textbf{Kan extension} of $d$ \textbf{along} $k$. We also call a (left) right Kan extension of $d$ along $k$ a $k$-\textbf{(co)limit}.
       \item We say that $c$ has all $k$-\textbf{(co)limits} and is $k$-\textbf{(co)complete} if every morphism $d'\colon I\to c$ has a (left) right Kan extension along $k$. When $K$ is a terminal object of $\K$, we identify $k$ with its domain $I$.
   \end{enumerate}
\end{definition}
For the formal theory of Kan extensions in $2$-categories, see \cite{streetYonedaStructures2categories1978}.
\begin{definition}
   Let $F = (F_i\colon\C\to \D_i)_{i\in I}$ be a family of functors, which we identify with the induced functor $F\colon \C\to \D\coloneqq \prod_i\D_i$. We say that $F$
   \begin{enumerate}
       \item \textbf{preserves} initial objects if $Fc$ is an initial object of $\D$ for each initial object $c$ in $\C$.
       \item \textbf{reflects} initial objects if an object $c$ in $\C$ is an initial object whenever $Fc$ is.
       \item \textbf{strictly creates} initial objects if for any initial object $d$ in $\D$ there is a unique object $c$ in $\C$ such that $Fc = d$, and, moreover, $c$ is an initial object in $\C$.
       \item \textbf{creates} initial objects if whenever $\D$ has an initial object, then $\C$ has one and $F$ preserves and reflects initial objects.
   \end{enumerate}
   We often use the word jointly to indicate that the indexing set $I$ might have cardinality different from $1$. These interaction properties have a dual form with respect to terminal objects.
\end{definition}
\begin{definition}[Preservation and Strict Creation of Kan Extensions]
   Consider a span $K\xleftarrow{k}I\xrightarrow{d}c\xrightarrow{f}c'$ of morphisms in a $2$-category $\K$. We say that $f$ \textbf{preserves, reflects, strictly creates} or \textbf{creates} the left Kan extension of $d$ along $k$, when the induced functor $f_*\colon \LE(d,k)\to \LE(fd,k), \eta\mapsto f*\eta,$ respects initial objects in the corresponding manner. Dually, we define the same notions for right Kan extensions. We say that $f$ preserves $k$-(co)limits if it preserves (left) right Kan extensions of morphisms $d'\colon I\to c$ along $k$.
\end{definition}

\begin{remark}
    If $\LE(d,k)$ has an initial object $(l,\eta\colon d\Rightarrow lk)$, then transposition along $\eta$ gives an isomorphism $\LE(d,k)\cong l\backslash \K(K,c)$ over $\K(K,c)$. Conversely, if $\LE(d,k)$ is equivalent to a coslice category, then it has an initial object, since every coslice category has one.
\end{remark}

\begin{lemma}\label{lem:PIEExtensionsFunctors}
    Let $\K$ be a $2$-category. Consider morphisms $K\xleftarrow{k} I\xrightarrow{d} a\xrightarrow{u} b\xrightarrow{f,g} c$ in $\K$. Then the following hold:
    \begin{enumerate}
        \item Let $(a_j)_{j\in J}$ be a family of objects in $\K$. Assume that the product $\prod_j a_j$ exists and equals $a$. Then we have the isomorphism of categories
        $$
        \LE(d,k) \cong \prod_{j\in J}\LE(d_j,k), (l = (l_j)_j,\eta = (\eta_j)_j)\mapsto (l_j,\eta_j)_j, \text{ where $d = (d_j)_j.$}
        $$
        \item Assume $(u\colon a\to b,\phi\colon fu\Rightarrow gu)$ is the inserter of $f$ and $g$, then there is an isomorphism of categories $F\colon \LE(d,k)\to \Ins(f_*,g_*)$ so that the diagram
        $$
        \begin{tikzcd}
{\LE(d,k)} \arrow[rr, "F"] \arrow[rd, "u_*"'] &             & {\Ins(f_*,g_*)} \arrow[ld, "U"] \\
                                              & {\LE(ud,k)} &
\end{tikzcd}
        $$
        commutes. Here $f_*,g_*\colon \LE(ud,k)\to \LE(fud,k)$ are functors defined on objects as follows:
        $$
        f_*(l'\colon K\to b,\eta'\colon ud\Rightarrow l'k) = (fl',f\eta'), \quad g_*(l',\eta') = (gl',g\eta'\bullet \phi d).
        $$
        \item Let $\alpha,\beta\colon f\Rightarrow g$ be transformations in $\K$ and assume $u\colon a\to b$ determines the equifier $\Eq(\alpha,\beta)$. Then there is an isomorphism $F\colon \LE(d,k)\cong \Eq(\alpha_*,\beta_*)$ in the slice category $\CAT/\LE(ud,k)$, where the horizontal composition defines the natural transformations $\alpha_*,\beta_*\colon f_*\Rightarrow g_*\colon \LE(ud,k)\to \LE(fud,k)$ ($g_*(l',\eta')=(gl',\,g\eta'\bullet\alpha ud)$).
    \end{enumerate}
\end{lemma}
\begin{proof}
    \hfill
    \begin{enumerate}
        \item The claim follows immediately from the universal property of products.
        \item Post-composing with $u$ determines a functor $F\colon \LE(d,k)\to \Ins(f_*,g_*)$, which is given on objects as follows:
        $$
        (l\colon K\to \Ins(f,g),\eta\colon d\Rightarrow lk)\mapsto ((ul,u\eta), f_*(ul,u\eta)\xrightarrow{\phi l} g_*(ul,u\eta)).
        $$
        We show that $F$ is an isomorphism of categories. The functor $F$ is bijective on objects since the objects in $\Ins(f_*,g_*)$ consist exactly of triples
        $$
        (l'\colon K\to b, \eta'\colon ud\Rightarrow l'k,  \gamma\colon (fl',f\eta')\to (gl',g\eta'\bullet \phi d)),
        $$
        each of which, by the universal properties of the inserters, corresponds uniquely to a pair $(l\colon K\to \Ins(f,g), \eta\colon d\Rightarrow lk)$ where $ul = l', \phi l = \gamma$ and $\eta' = u\eta$. The second universal property of inserters shows the full faithfulness of $F$. Thus $F$ is an isomorphism and, by construction, $UF = u_*$.

        \item Again by post-composing with $u$, we obtain a functor $F\colon \LE(d,k)\to \Eq(\alpha_*,\beta_*)$. As the object $(l\colon K\to \Eq(\alpha,\beta), \eta\colon d\Rightarrow lk)$ in $\LE(d,k)$ naturally and bijectively corresponds to a pair $(l'\colon K\to b, \eta'\colon ud\Rightarrow l'k)$, where $\alpha l' = \beta l'$, via $(l,\eta)\mapsto (ul,u\eta)$, we have the bijectivity on objects of the functor $F\colon \LE(d,k)\to \Eq(\alpha_*,\beta_*)$. The full faithfulness follows similarly.\qedhere
    \end{enumerate}
\end{proof}

\subsection{Lifting Kan Extensions Along PIE Limits}
In this subsection, we fix a signature pair $(\Sigma,\sigma)$ with a $\sigma$-theory $T$. The following theorem serves as the primary technical result of this section, demonstrating how the existence of Kan extensions lifts along PIE limits. This theorem provides a unified $2$-categorical framework for limit and colimit inheritance. Consequently, classical results showing that models of universal algebraic theories\footnote{Universal algebraic theory $T$ refers, from our perspective, to an algebraic theory whose associated $2$-signature is the cartesian signature $\Sigma_{\text{Cart}}$ in Example~\ref{ex:2-signatures}.}, Eilenberg--Moore categories, and categories of coalgebras inherit structures from their background categories follow as immediate consequences of these three basic $2$-categorical constructions.

\begin{theorem}[Lifting Kan Extensions Along PIE Limits]\label{thm:pie_kan_creation}
   Let $\K$ be a $2$-category. Then the following assertions hold:
   \begin{enumerate}
       \item Let $(c_i)_{i \in J}$ be a family of objects in $\K$ whose product $c = \prod_{i\in J} c_i$ exists in $\K$. Then the family $(\pi_i \colon c \to c_i)_{i\in J}$ of projections preserves and jointly strictly creates all right and left Kan extensions. Additionally, the family of projections jointly strictly creates and preserves (subterminality) quotient initiality of a (right) left extension.
       \item Consider the inserter diagram
       \[
\begin{tikzcd}
{\Ins(f,g)} \arrow[d, "u"'] \arrow[r, "u"]        & a \arrow[d, "g"] \\
a \arrow[r, "f"'] \arrow[ru, "\phi" description, Rightarrow] & b
\end{tikzcd}
\]
       in $\K$. Then the morphism $u$ strictly creates left Kan extensions that $f$ preserves, and strictly creates right Kan extensions that $g$ preserves.
       \item Let $\Eq(\eta,\theta)\xrightarrow{u}c$ be the equifier of $\eta,\theta \colon f \Rightarrow g \colon c \to c'$. Then $u$ strictly creates left (right) Kan extensions whose quotient initiality (subterminality) is preserved by $f$ $(g)$.
   \end{enumerate}
\end{theorem}
\begin{proof} We show the claims for left Kan extensions; the cases for right Kan extensions follow dually.
   \begin{enumerate}

   \item Let $d = (d_i)_i\colon I\to \prod_i c_i$ and $k\colon I\to K$. As we have the isomorphism
   $$
   \LE(d,k)\cong\prod_i \LE(d_i,k), ((l_i)_i,(\eta_i)_i)\mapsto (l_i,\eta_i)_i,
   $$
   we have that $\LE(d,k)$ has an initial object if and only if each category $\LE(d_i,k)$ has one. Via this isomorphism, the projections both preserve and strictly create all left Kan extensions. The case for quotient initiality holds similarly.

    \item Let $F,G\colon \C\to \D$ be functors. Via Lemma \ref{lem:PIEExtensionsFunctors}(2), it suffices to show that the forgetful functor $U\colon \Ins(F,G)\to \C$ strictly creates those initial objects that $F$ preserves. Assume $0$ and $F0$ are initial objects in $\C$ and $\D$, respectively. Clearly, $(0, !\colon F0\to G0)$ is the unique object in the fiber of $0$ along $U$. As $(0,!\colon F0\to G0)$ clearly is the initial object in $\Ins(F,G)$, $U$ strictly creates those initial objects that $F$ preserves.

       \item Assume $\alpha,\beta\colon F\Rightarrow G\colon \C\to \D$ are natural transformations between functors. By Lemma \ref{lem:PIEExtensionsFunctors}(3), it suffices to show that the forgetful functor $U\colon \Eq(\alpha,\beta)\to \C$ strictly creates those initial objects whose quotient initiality is preserved by $F$. Assume $0$ is an initial object in $\C$ and $F0$ is a quotient initial object in $\D$. It is sufficient to show that $\alpha_0 =\beta_0$, but this is clear as $F0$ is a quotient initial object. \qedhere
   \end{enumerate}
\end{proof}

For $\K = \CAT$, the following corollary generalizes the limit lifting results found in \cite[3.3, 3.4]{Ulmer1977Bialgebras}:
\begin{corollary}\label{cor:strictCreationToBialgebraicDatumObjectsKan}
    Let $\B = (c,\ f,g,\ f',g',\ \alpha,\beta)$ be a bialgebraic datum in $\K$. Then the canonical morphism $u_\B\colon \BialgB\to c$ strictly creates those left Kan extensions that $f$ preserves and whose quotient initiality $f'$ preserves. Dually, $u_\B$ strictly creates those right Kan extensions that $g$ preserves and whose subterminality is preserved by $g'$.
\end{corollary}
\begin{proof}
    Follows directly from parts (2) and (3) of Theorem \ref{thm:pie_kan_creation}.
\end{proof}

\begin{definition}
     We call a bialgebraic datum $\B = (c,\ f,g,\ f',g',\ \alpha,\beta)$ \textbf{relatively algebraic} if $g$ preserves all right Kan extensions and $g'$ preserves their subterminality. Dually, we define when $\B$ is \textbf{relatively coalgebraic}. We say a bialgebraic theory $T$ is $M$-\textbf{relatively (co)algebraic} if the bialgebraic datum $\B_{M,T}$ is relatively (co)algebraic.
\end{definition}

A left Kan extension in an object $c$ of $\K$ is \textbf{absolute} if each morphism $c\to d$ preserves it. For morphisms $K\xleftarrow{k}I\xrightarrow{d}c\xrightarrow{f}c'$, we say that a left extension $(l,\eta)$ in $\LE(d,k)$ is an \textbf{$f$-absolute $k$-colimit} if $(fl,f\eta)$ is an absolute left Kan extension of $fd$ along $k$.
\begin{corollary}\label{cor:LiftingAbsoluteExtensions}
    For a bialgebraic datum $\B$ over $c$, the canonical morphism $u\colon \BialgB\to c$ strictly creates $u$-absolute $k$-limits and $k$-colimits for every morphism $k$ in $\K$.
\end{corollary}
\begin{proof}
    Since all morphisms preserve absolute Kan extensions, $u$ strictly creates $u$-absolute Kan extensions (Corollary \ref{cor:strictCreationToBialgebraicDatumObjectsKan}).
\end{proof}

\noindent
For the rest of this subsection, we will fix a PIE-complete $2$-category $\K$ with a $\Sigma$-model $M$.

\begin{corollary}\label{cor:LimitsLiftedToBialgebras}
   Consider the canonical morphism $M^T\xrightarrow{u} M_0$. Then $u$ strictly creates those left Kan extensions that $M(t_1)$ preserves and whose quotient initiality $M(t_1')$ preserves for each function arity term $t_1$ and equational arity term $t_1'$. Dually, $u$ strictly creates those right Kan extensions that $M(t_2)$ preserves and whose subterminality $M(t_2')$ preserves for each function coarity term $t_2$ and equational coarity term $t_2'$.
\end{corollary}
\begin{proof}
   The induced bialgebraic datum $\B_{M,T}$ (Proposition~\ref{prp:bialgebraic_datum}) satisfies the assumptions of Corollary~\ref{cor:strictCreationToBialgebraicDatumObjectsKan} by Theorem~\ref{thm:pie_kan_creation}(1).
\end{proof}

Theorem \ref{thm:pie_kan_creation}(1) shows that all (co)algebraic theories $T$ are $M$-relatively (co)algebraic in any $\Sigma$-model $M$. Corollary \ref{cor:strictCreationToBialgebraicDatumObjectsKan} shows, for any morphism $k$ in $\K$, that the $k$-(co)limits of $\BialgB$ are strictly created via the canonical morphism, when $\B$ is relatively (co)algebraic.

\begin{corollary}\label{cor:limitPreservationFromAtomicPreservation}
Let $k\colon I\to K$ be a morphism in $\K$. Assume each morphism $M(B)$ preserves $k$-colimits. Let $t\colon c$ be a term in $\Term_S^\Sigma$. Then $M(t)\colon M_0\to M_c$ preserves all left Kan extensions along $k$. In particular, the canonical morphism $u\colon M^T\to M_0$ strictly creates $k$-colimits.
\end{corollary}
\begin{proof}
   An easy structural induction shows that $M(t)$ preserves $k$-colimits for each term $t$. Thus, Corollary~\ref{cor:LimitsLiftedToBialgebras} yields that $M^T\to M_0$ strictly creates $k$-colimits.
\end{proof}

We list some classical results that follow from our framework.
\begin{corollary}
   Let $T$ be a monad on a category $\C$. Then the forgetful functor $\C^T\to \C$ from the Eilenberg--Moore category $\C^T$ of $T$ strictly creates limits and those colimits that $T$ preserves and whose quotient initiality $T^2$ preserves.
\end{corollary}
\begin{proof}
    Consider the theory of monad algebras. Since the function arity support is $T$ and the equational arity support is $(\id_\C, T^2)$, Corollary~\ref{cor:LimitsLiftedToBialgebras} yields the result.
\end{proof}

\begin{corollary}
   Let $F\colon \C\to \E$ and $G\colon \D\to \E$ be functors. Then the forgetful functor $U\colon F\downarrow G\to \C\times \D$ strictly creates those colimits and limits that $F$ and $G$ preserve, respectively.
\end{corollary}
\begin{proof}
    Note that $F\downarrow G\cong \Ins(F\pi_1,G\pi_2)$ over $\C\times \D$ and $F\pi_1$ preserves a colimit in $\C\times\D$ if $F$ preserves it in the first component (Theorem \ref{thm:pie_kan_creation}(2)). The dual case holds for $G\pi_2$.
\end{proof}

\begin{corollary}
   Let $c$ be an object of a category $\C$. Then the forgetful functor $U\colon \C/c\to \C$ from the slice category over $c$ strictly creates all colimits and connected limits.
\end{corollary}
\begin{proof}
   The theory for $\C/c$ is coalgebraic, which yields the creation of colimits. Since a constant functor $\C\to 1\to \C$ is the only coarity functor and it preserves all connected limits, $\C/c\to \C$ strictly creates connected limits.
\end{proof}

\subsection{Sifted Colimits and Pushouts of Families of Regular Epis}\label{subsec:sifted}

We now specialize to the $2$-category $\CAT$, where the preservation hypotheses of Corollary~\ref{cor:limitPreservationFromAtomicPreservation} can be checked componentwise for sifted colimits.

A category $I$ is \textbf{sifted} if it is non-empty and the diagonal $\Delta\colon I\to I\times I$ is final \cite{AdamekRosickyVitale2010SiftedColimits, Adamek-Rosicky-Vitale}.\footnote{Recall that a functor $F\colon \C\to \D$ is final if $d\downarrow F$ is a connected (implying non-empty) category for each object $d$ in $\D$ and equivalently for any diagram $H\colon \D\to \E$ the functor of cocones $F^*\colon \mathrm{Cocone}(H)\to \mathrm{Cocone}(HF), (\eta\colon H\Rightarrow e)\mapsto (\eta F\colon HF\Rightarrow e),$ is an isomorphism of categories.} We recall some important examples of sifted categories:
\begin{enumerate}
    \item A category $I$ is called \textbf{filtered} if it is non-empty and any finite diagram $D$ on $I$ admits a cocone $D\Rightarrow i$. If a poset, as a category, is filtered, it is called (upward) directed. A colimit of a functor $I\to \C$ is called a filtered colimit.
    \item Fix the category $I \coloneqq \begin{tikzcd}
\bullet \arrow[r, shift left=2] \arrow[r, shift right=2] & \bullet \arrow[l]
\end{tikzcd}$, the shape of \textbf{reflexive pairs}: a functor out of it chooses exactly a triple $(r_1,r_2\colon x\to y; s\colon y\to x)$ of morphisms where $s$ is a common section, $r_1 s = r_2s = \id_y$. The colimit of a functor $I\to \C$ is called a reflexive coequalizer.
    \item We may generalize the previous example and consider a category $I = (\begin{tikzcd}
k \arrow[r, shift left=2] \arrow[r, shift right=2] & a \arrow[l]
\end{tikzcd})_{k\in K}$ of $K$-many reflexive pairs on an object $a$, where $K$ is a non-empty set. The category $I$ is sifted and a colimit of a diagram $D\colon I\to \C$ generalizes the notion of a wide pushout of $K$-many regular epimorphisms with a common domain under suitable assumptions on $\C$: If each regular epimorphism in $\C$ is a coequalizer of a reflexive pair, then a $K$-wide pushout of a family $(f_i\colon c\to y_i)_{i\in K}$ of regular epis with a common domain $c$ can be realized as a sifted colimit. This is done by associating to $f_i$ a reflexive pair, whose coequalizer is $f_i$ for each $i$. This defines a functor $D\colon I\to \C$. The wide pushout is then computed as the (sifted) colimit of $D$. Note that a functor preserving reflexive coequalizers preserves this $K$-wide pushout if and only if it preserves the colimit of $D$.

Dualizing this example, we obtain a way to compute the intersection of regular subobjects via cosifted limits, when the ambient category either has binary products or cokernel pairs.
\end{enumerate}

\begin{lemma}\label{lem:SiftedTotalityFromComponents}
   Let $F\colon \C \to \D$ be a functor, where $\C = \C_1\times\cdots \times \C_n$. Let $D = (D_i)_i\colon I\to \C$ be a diagram with a colimit and assume $I$ is a sifted category. Then $F$ preserves the colimit of $D$ if $F$ preserves the colimit of $D$ componentwise, meaning for any object $(c_i)_i$ in $\C$ the functor $x_i\mapsto F(c_1,\ldots,x_i,\ldots, c_n)\colon \C_i\to \D$ preserves the colimit of $D_i$ for $i\leq n$.
\end{lemma}
\begin{proof}
   Assume $F$ preserves the colimit of $D$ componentwise. We show that $F$ preserves the colimit of $D$. It suffices to prove the claim for $n = 2$; the general case follows by induction. For $n = 2$, we have
   \begin{align*}
       F(\colim_i (D_1(i), D_2(i)))
       &\cong F(\colim_{i,j}(D_1(i),D_2(j)))\tag{$\Delta$ is final}\\
       &\cong F(\colim_i \colim_j(D_1(i),D_2(j)))\tag{Fubini}\\
       &\cong \colim_i F(\colim_j(D_1(i),D_2(j)))\tag{Preservation componentwise}\\
       &\cong \colim_i\colim_j F(D_1(i),D_2(j))\tag{Preservation componentwise}\\
       &\cong \colim_{i,j} F(D_1(i),D_2(j))\tag{Fubini}\\
       &\cong \colim_i F(D_1(i),D_2(i)).\tag{$\Delta$ is final}
   \end{align*}
\end{proof}

The following corollary highlights the ubiquity of sifted colimits and cosifted limits in practice. Whenever a monoidal category possesses reflexive coequalizers that are preserved componentwise by its tensor product, any category of internal bialgebras constructed over it strictly inherits these reflexive coequalizers.

\begin{corollary}\label{cor:liftingSifted}
   Let $(\Sigma, \sigma)$ be a signature pair with a $\sigma$-theory $T$. Let $M$ be a categorical $\Sigma$-model and $I$ be a sifted category. If $M(B)$ preserves $I$-colimits componentwise for each $\Sigma$-functor symbol $B$, then the forgetful functor $U\colon M^T\to M_0$ strictly creates all $I$-indexed colimits.
\end{corollary}
\begin{proof}
   This follows immediately from Lemma~\ref{lem:SiftedTotalityFromComponents} and Corollary~\ref{cor:limitPreservationFromAtomicPreservation}.
\end{proof}

When a category has binary coproducts or kernel pairs, each regular epimorphism is a coequalizer of a reflexive pair. This allows one to witness a wide pushout of regular epimorphisms as a sifted colimit. In the following we show how this realization allows one to compute certain kinds of pushouts in bialgebraic structures by computing them in the underlying structure.
\begin{corollary}\label{cor:fromComponentwisePreservationOfRegSubobjects}
    Let $\otimes\colon \C_1\times\cdots\times \C_n\to \C$ be a functor. Assume for each $i\leq n$ that every regular epimorphism in $\C_i$ is a coequalizer of a reflexive pair. Let $I$ be a set of cardinality at least $2$. If $\otimes$ preserves $I$-indexed pushouts of regular epimorphisms componentwise, then it preserves such pushouts and regular epimorphisms.
\end{corollary}
\begin{proof}
    Assume that $\otimes$ preserves $I$-indexed pushouts of regular epis componentwise. We may also assume that $n = 2$. Note that a functor preserving pushouts of regular epis along regular epis preserves coequalizers of reflexive pairs. Consider an $I$-indexed family $f = ((f^i_1,f^i_2)\colon (a_1,a_2)\to (a_1^i,a_2^i))_{i\in I}$ of regular epimorphisms in $\C_1\times \C_2$. By assumption, there is a sifted diagram $D\colon \bar{I}\to \C_1\times \C_2$ that chooses reflexive pairs on $(a_1,a_2)$ whose coequalizers produce the morphisms $(f^i_1,f^i_2)_i$. Since $\otimes$ preserves the $I$-indexed pushouts componentwise, it preserves the colimit of $D$ componentwise and hence the colimit of $D$ (Lemma \ref{lem:SiftedTotalityFromComponents}). Here we used the fact that $\otimes$ preserves reflexive coequalizers componentwise. The preservation of the colimit of $D$ shows that $\otimes$ preserves the $I$-indexed pushout $(f_1^i,f^i_2)_i$ in $\C_1\times \C_2$. Lastly, as $\otimes$ preserves $I$-indexed pushouts, it preserves pushouts of regular epis along regular epis and especially reflexive coequalizers. Hence $\otimes$ preserves regular epimorphisms, since by assumption they can be written as coequalizers of reflexive pairs.
\end{proof}

\begin{theorem}\label{thm:creation_wide_pushouts}
    Let $(\Sigma,\sigma)$ be a signature pair,  $T$ a $\sigma$-theory, $M$ a $\Sigma$-model in $\CAT$, and $I$ a set with cardinality at least $2$. Assume in each category $M_c$ regular epimorphisms are reflexive coequalizers and each functor $M(B)$ preserves $I$-indexed wide pushouts of regular epis componentwise. Let $f_i\colon a\to b_i$ be morphisms in $M^T$ such that $Uf_i$ is a regular epimorphism for each $i\in I$. Then the forgetful functor $U\colon M^T\to M_0$ strictly creates the $I$-indexed pushout of $(f_i)_{i\in I}$.
\end{theorem}
\begin{proof}
    By Corollary~\ref{cor:fromComponentwisePreservationOfRegSubobjects} the functors $M(B)$ preserve regular epis and $I$-indexed pushouts of regular epis. Inductively, every term functor $M(t)$ preserves $I$-indexed pushouts of regular epis and the claim follows by Corollary~\ref{cor:LimitsLiftedToBialgebras}.
\end{proof}

\begin{example}
    Consider a symmetric monoidal closed category $\C = (\C,\otimes,I)$ that is both complete and cocomplete. Assume the tensor $\otimes$ preserves finite limits componentwise. Consider the category $\Hopf$ of Hopf algebras in $\C$. Corollary~\ref{cor:liftingSifted} allows one to compute the following limits and colimits in $\Hopf$ by strictly creating them from $\C$:
    \begin{enumerate}
        \item Reflexive coequalizers and coreflexive equalizers in $\Hopf$ are strictly created from $\C$ as the tensor preserves such (co)limits (Corollary \ref{cor:liftingSifted} and its dual).
        \item More generally, consider morphisms $f_i\colon x_i\to y$ in $\Hopf$ that are regular monomorphisms in $\C$ for $i\leq n$. The joint pullback of the $f_i$ is strictly created from $\C$ (dual of Theorem \ref{thm:creation_wide_pushouts}). Similarly, any wide pushout of morphisms in $\Hopf$ that are regular epimorphisms in $\C$ is strictly created from $\C$. This explains why for a field $k$ one can compute finite intersections of Hopf subalgebras in $\Hopf_k$ in the underlying category $\Vect_k$ of $k$-vector spaces.
        \item Any filtered colimit in $\Hopf$ is strictly created from $\C$.
    \end{enumerate}
    Of course, there is nothing special about $\Hopf$. We could have chosen any bialgebraic theory $T$ and considered $\C^T$, and obtained the same results. Similarly, nothing depends on the symmetric monoidal structure: we could have chosen any $\Sigma$-model $M$, where the categories $M_c$ have and the functors $M(B)$ (componentwise) preserve sifted colimits and finite cosifted limits, and the categories $M_c$ have kernel pairs and cokernel pairs.
\end{example}

\begin{corollary}
    The category $\mathbf{Field}$ of fields is accessible, and the forgetful functor $\mathbf{Field}\to \mathbf{Set}, k\mapsto k^*,$ creates connected limits and sifted colimits.
\end{corollary}
\begin{proof}
    Consider the signature pair $(\Sigma,\sigma)$ with the theory of fields $T$ and the categorical $\Sigma$-model $M = (\Set, \times, 1, (-)+1)$ from Example~\ref{exs:theories}(3). As the structural functors of $M$ preserve connected limits and sifted colimits, Corollary~\ref{cor:limitPreservationFromAtomicPreservation} shows that $M^T\simeq \textbf{Field}\to \Set$ creates connected limits and sifted colimits. As the category $\Set$ is accessible and the structural functors are accessible (they preserve directed colimits), it follows that $M^T$ is accessible.
\end{proof}

\section{Detection of Colimits and Monadicity}\label{sec:monadicity}

In this section, we work in the $2$-category $\CAT$. We show that, under suitable assumptions on colimits, the forgetful functor $M^T\to M_0$ is monadic for algebraic theories $T$, and that it detects Kan extensions, in particular colimits. Comonadicity for coalgebraic theories requires accessibility or local presentability assumptions (compare the duality between varieties and covarieties in \cite{adamekVarietiesCovarietiesCategory2003}), which we treat in Subsection~\ref{subsec:comonadicity}.

\subsection{Monadicity in a \texorpdfstring{$2$}{2}-Category}
\begin{definition}
    A morphism $g\colon d\to c$ in a $2$-category $\K$ is said to be \textbf{monadic} if $g$ admits a left adjoint $f$ with $(\eta,\varepsilon)$ as the unit-counit pair, where the triple $(d,g\colon d\to c,g\varepsilon\colon gfg\Rightarrow g)$ is an Eilenberg--Moore object of the monad $(gf,g\varepsilon f,\eta)$ up to equivalence, that is, the comparison functor $\K(x,d)\to\K(x,c)^{\K(x,gf)}$, $h\mapsto (gh,g\varepsilon h)$, is an equivalence of categories for each object $x$ in $\K$. \textbf{Comonadicity} is defined dually.
\end{definition}
Beck's monadicity theorem shows that a right adjoint functor $G\colon \D\to\C$ is monadic if and only if $G$ creates $G$-split coequalizers. This means that $G$ creates the coequalizer of $f,g\colon c\to d$ given that the pair $G(f),G(g)$ has a split coequalizer. A split coequalizer of $f,g\colon c\to d$ is a triple $(r\colon d\to e,s\colon e\to d, j\colon d\to c)$ so that $rf = rg$, $rs = \id_e$, $gj = \id_d$ and $fj = sr$.
It should be noted that $r$ is then an absolute coequalizer of $f$ and $g$.
\begin{lemma}\label{lem:2DimMonadicity}
    Let $g\colon d\to c$ be a morphism in a $2$-category $\K$. Then $g$ is monadic if and only if $g$ is a right adjoint and $g_x\colon \K(x,d)\to \K(x,c)$ creates $g_x$-split coequalizers for each object $x$ in $\K$.
\end{lemma}
\begin{proof}
    Assume $g$ has a left adjoint $f\dashv g$ with counit $\varepsilon$. Then $f_x\dashv g_x$ for each object $x$, and the induced monad on $\K(x,c)$ is $\K(x,gf)$. By definition, $g$ is monadic if and only if the comparison functor $\K(x,d)\to\K(x,c)^{\K(x,gf)}$ of $g_x$ is an equivalence for each $x$. By Beck's monadicity theorem, the latter holds if and only if $g_x$ creates $g_x$-split coequalizers. Thus the claim holds.
\end{proof}

\begin{corollary}\label{cor:GeneralMonadicity}
    Consider the following commutative diagram in a $2$-category $\K$:
    $$
    \begin{tikzcd}
a \arrow[rd, "u"'] \arrow[rr, "g"] &   &b \arrow[ld, "v"] \\
                                       & c &
\end{tikzcd}
    $$
    Assume $g$ is a right adjoint and, for each object $x$ in $\K$, the functor $v_x\colon \K(x,b)\to\K(x,c)$ is conservative and $u_x\colon \K(x,a)\to\K(x,c)$ creates $u_x$-split coequalizers. Then $g$ is monadic. In particular, for any bialgebraic data $\B$ and $\B'$ in $\K$ over the same object $c$ in $\K$, a (left) right adjoint $g\colon \BialgB\to \Bialg(\B')$, where $u_{\B'}g = u_\B$, is (co)monadic.
\end{corollary}
\begin{proof}
    Consider an object $x$ in $\K$. By Lemma~\ref{lem:2DimMonadicity}, it suffices to show that $g_x\colon \K(x,a)\to \K(x,b)$ creates $g_x$-split coequalizers. Let $d_1,d_2$ be a $g_x$-split coequalizer pair in $\K(x,a)$. Then $(vg)_x(d_1),(vg)_x(d_2)$ also has a split coequalizer, which is then created by $u_x$. Thus $d_1,d_2$ has a coequalizer and it is preserved by $u_x$. By conservativity of $v_x$, $g_x$ also preserves this coequalizer. Similarly, $g_x$ reflects the coequalizers of $d_1$ and $d_2$. Thus $g_x$ creates $g_x$-split coequalizers.
\end{proof}

\subsection{Free Algebras and Detection of Kan Extensions}
We say that a morphism $u\colon c\to c'$ detects a left/right Kan extension of $d\colon I\to c$ along $k\colon I\to K$ via morphisms $f_i\colon c'\to x_i$ if the existence of left/right Kan extensions of $ud$ and $f_iud$ along $k$ for all $i$ implies that $d$ has a left/right Kan extension along $k$.
\begin{definition}
    Let $\kappa$ be an infinite cardinal. As $\kappa$ is well-ordered, we identify $\kappa$ with the standard category determined by its posetal structure. A functor $F\colon\kappa\to \C$ is called a $\kappa$-\textbf{chain} in $\C$. A colimit of a $\kappa$-chain is called a $\kappa$-\textbf{colimit}, and $\C$ is $\kappa$-\textbf{cocomplete} if it has all $\kappa$-colimits. We say that $\C$ is $\kappa_\leq$-\textbf{cocomplete} if $\C$ is $\lambda$-cocomplete for each infinite cardinal $\lambda\leq \kappa$.
\end{definition}

\begin{theorem}\label{thm:BigMonadicity}
Let $\kappa$ be an infinite cardinal and let $F,G\colon \C\to \D$ be functors with $\C$ $\kappa_\leq$-cocomplete. Assume that $G$ has an $F$-relative left adjoint $L$.\footnote{An $F$-relative left adjoint $L$ for $G$ means that there is an isomorphism $\hom(Lx,y)\cong\hom(Fx,Gy)$ natural in $x$ and $y$. If $G$ has a left adjoint $l$, then $L \cong lF$ \cite{ulmerPropertiesDenseRelative1968}} Then the following assertions hold:
\begin{enumerate}
    \item If $\C$ has finite coproducts and $F$ preserves $\kappa$-colimits, then $U\colon \Ins(F,G)\to \C$ is monadic.
    \item Let $\alpha,\beta\colon F\Rightarrow G$ be natural transformations. If $\C$ has coequalizers and $F$ preserves joint epicness of colimiting cocones of $\kappa$-chains, then the forgetful functor $U\colon \Eq(\alpha,\beta)\to \C$ is monadic.
\end{enumerate}
\end{theorem}
\begin{proof}
Without loss of generality, we may assume that $G = \id_\C\colon\C\to \C$, since the isomorphism $\hom(Lx,y)\cong \hom(Fx,Gy)$, natural in $x$ and $y$, determines an isomorphism $\Ins(L,\id_\C)\cong \Ins(F,G)$ over $\C$. Moreover, $L$ preserves those colimiting cocones that $F$ preserves and $L$ also preserves joint epicness of those families of morphisms whose joint epicness $F$ preserves. This justifies the assumption $G = \id_\C$. By Corollary \ref{cor:GeneralMonadicity}, to show the monadicity of the canonical forgetful functor from the object of bialgebras, it suffices to show only that it is a right adjoint.
\begin{enumerate}
    \item Assume that $\C$ has finite coproducts. Let $c$ be an object of $\C$. Consider the functor $W = c+F(-)\colon \C\to \C$, which preserves those connected colimits that $F$ preserves. A standard argument transforms the initial algebra of $W$ (the initial object of $\Ins(W,\id_\C)$) into the free object of $\Ins(F,\id)$ on $c$ \cite{Adamek1974}. Define
    $$
    W_{\lambda} = \begin{cases}
            \colim_{\alpha<\lambda}W_{\alpha}, &\text{ $\lambda$ a limit ordinal (0 is a limit ordinal)},\\
        W(W_\alpha), &\text{ if $\lambda = \alpha +1$}
    \end{cases}
    $$
    for $\lambda\leq \kappa$ with the induced morphisms defining a functor $W_{(-)}\colon \kappa\to \C$. As $W$ preserves colimits of $\kappa$-chains, there is an induced isomorphism $W(W_\kappa)\to W_\kappa$ which determines the initial algebra of $W$, which shows that $U$ is a right adjoint. A full proof can be found in \cite{Adamek1974}.

    \item Assume that $\C$ has coequalizers. We show that the forgetful functor $U\colon \Eq(\alpha,\beta)\to \C$ has a left adjoint. Let $c$ be an object of $\C$. We define $W\colon \kappa\to \C$ by first setting $W_0 = c$. Inductively assume that $W|_{\lambda}\colon \lambda\to \C$ is defined for some $0<\lambda<\kappa$. We extend $W$ to the value $\lambda$ as follows:
    {\small $$
    W_{\lambda',\lambda} = \begin{cases}
        \mathrm{Coeq}(\alpha_{W_{\lambda''}},\beta_{W_{\lambda''}}\colon FW_{\lambda''}\to W_{\lambda''})\circ W_{\lambda',\lambda''}\colon W_{\lambda'}\to W_{\lambda}, &\text{ if $\lambda'\leq\lambda'', \lambda = \lambda''+1$}\\
        W_{\lambda'}\xrightarrow{\text{can}_{\lambda'}}\colim_{{\lambda''}<\lambda}W_{\lambda''}, & \text{ if $\lambda>0$ is a limit ordinal.}
    \end{cases}
    $$
    }
    for $\lambda< \kappa$. Consider the colimiting cocone $\eta\colon W\Rightarrow d\colon \kappa\to \C$. First, we show that $\alpha_d = \beta_d$. By assumption, $F\eta\colon FW\Rightarrow Fd$ is a jointly epic cocone and so it suffices to show that $\alpha_d\bullet F\eta_\lambda = \beta_d\bullet F\eta_\lambda$ for all $\lambda$. Now for any $\lambda< \kappa$, we have
    $$
    \alpha_d\circ F\eta_\lambda = \eta_\lambda \circ\alpha_{W_\lambda} = \eta_{\lambda+1}\circ W_{\lambda,\lambda+1}\circ \alpha_{W_\lambda} = \eta_{\lambda+1}\circ W_{\lambda,\lambda+1}\circ \beta_{W_\lambda} = \beta_d\circ F\eta_\lambda.
    $$
    Thus $\alpha_d = \beta_d$ and so $d$ is an object in $\Eq(\alpha,\beta)$. Next, we show that $\eta_0\colon c\to d$ is an initial object of $c\downarrow U$. Let $f\colon c\to Ue$ be an object in $c\downarrow U$. We need to show that there is a unique morphism $\bar{f}\colon d\to e$ in $\C$ satisfying $\bar{f}\eta_0 = f$. Uniqueness is clear since an induction argument shows that $\eta_0$ is an epimorphism. It remains to construct a natural transformation $\theta\colon W\Rightarrow e$ with $\theta_0 = f$; the colimit then yields the extension $\bar{f}\colon d\to e$ of $f$. Assume $\theta_{\lambda'}$ is defined for $\lambda'< \lambda$. If $\lambda$ is a limit ordinal, then as $W_\lambda = \colim_{\lambda'< \lambda} W_{\lambda'}$, we define $\theta_\lambda\colon W_\lambda\to e$ as the morphism induced by the natural transformation $\theta|_\lambda\colon W|_\lambda\Rightarrow e$. Assume then that $\lambda = \lambda'+1$. Now
    \[
    \theta_{\lambda'} \alpha_{W_{\lambda'}} = \alpha_e F(\theta_{\lambda'}) = \beta_e F(\theta_{\lambda'}) = \theta_{\lambda'}\beta_{W_{\lambda'}}.
    \]
    As $W_{\lambda',\lambda}$ is the coequalizer of $\alpha_{W_{\lambda'}}$ and $\beta_{W_{\lambda'}}$, there is a unique morphism $\theta_{\lambda}\colon W_\lambda\to e$ such that $\theta_\lambda W_{\lambda',\lambda} = \theta_{\lambda'}$. Therefore, we obtain the unique natural transformation $\theta\colon W\Rightarrow e$ extending $f$, which produces the extension $\bar{f}\colon d\to e$ via the colimit.\qedhere
\end{enumerate}
\end{proof}

\begin{lemma}\label{lem:KanExtensionInserter}
Let $\K$ be a $2$-category with a morphism $f\colon c\to c$. Consider the inserter diagram
$$
\begin{tikzcd}
{\Ins(f,\id_c)} \arrow[d, "u"'] \arrow[r, "u"] & c \arrow[d, "\id_c"] \\
c \arrow[r, "f"'] \arrow[ru, "\phi" description, Rightarrow] & c
\end{tikzcd}
$$
in $\K$. Let $K\xleftarrow{k}I\xrightarrow{d} \Ins(f,\id_c)$ be a span in $\K$. Assume that $\K(K,c)$ has pushouts and the left Kan extensions $(l_1,\eta_1\colon ud\Rightarrow l_1k)$ and $(l_2,\eta_2\colon fud\Rightarrow l_2k)$ of $ud$ and $fud$ along $k$ exist. Consider the pushout induced endofunctor
$$
H = (l_1,\eta_1)+_{l_2} f_*(-)\colon \LE(ud,k)\to \LE(ud,k).
$$
Then there is an isomorphism of categories $\LE(d,k) \cong \Ins(H,\id)$ over $\LE(ud,k)$ with $H$ preserving those connected colimits that exist in $\K(K,c)$ and are preserved by $f_K\colon \K(K,c)\to \K(K,c)$. In particular, if $\K(K,c)$ has $\kappa_\leq$-colimits and $f_K$ preserves colimits of $\kappa$-chains for some infinite cardinal $\kappa$, then $u_*\colon \LE(d,k)\to \LE(ud,k)$ has a left adjoint and $d\colon I\to \Ins(f,\id_c)$ has a left Kan extension along $k$.
\end{lemma}

\begin{proof}
We define $\alpha\colon l_2\Rightarrow l_1$ and $\alpha'\colon l_2\Rightarrow fl_1$ as the respective transpositions of
\[
fud\xRightarrow{\phi d} ud\xRightarrow{\eta_1}l_1k \quad \text{and}\quad fud\xRightarrow{f\eta_1}fl_1k.
\]
The forgetful functors induce the following isomorphisms:
$$
\LE(ud,k)\cong l_1\backslash\K(K,c), \quad \LE(fud,k)\cong l_2\backslash\K(K,c).
$$
By Lemma~\ref{lem:PIEExtensionsFunctors}(2), $u_*\colon \LE(d,k)\to \LE(ud,k)$ is isomorphic to the forgetful functor $\Ins(f_*,\id_*)\to \LE(ud,k)$ with the identifications of $\LE(ud,k)$ and $\LE(fud,k)$ as coslice categories:
$$
f_*(l_1\Rightarrow l') = (l_2\Rightarrow fl_1\Rightarrow fl'),\quad \id_*(l_1\Rightarrow l')= (l_2\xRightarrow{\alpha} l_1\Rightarrow l').
$$
As $\id_*\colon l_1\backslash\K(K,c)\to l_2\backslash\K(K,c)$ is the functor $\alpha^*$ induced by the morphism $\alpha\colon l_2\Rightarrow l_1$ in $\K(K,c)$, we have that $\id_*$ has a left adjoint $\alpha_*\colon \LE(fud,k)\to \LE(ud,k)$ defined by pushouts along $\alpha$. Thus $u_*$ is isomorphic to the forgetful functor $\Ins(\alpha_*f_*,\id)\to \LE(ud,k)$ in the slice category $\CAT/\LE(ud,k)$. The functor $H$ in the statement is the functor $\alpha_*f_*\colon \LE(ud,k)\to \LE(ud,k)$. Since the forgetful functors $l_i\backslash \K(K,c)\to \K(K,c)$, for $i = 1,2$, strictly create connected colimits, $f_*\colon l_1\backslash\K(K,c)\to l_2\backslash \K(K,c)$ preserves those connected colimits that exist in $\K(K,c)$ and are preserved by $f_K$ and thus the same connected colimits are preserved by $H=\alpha_*f_*$ as $\alpha_*$ is a left adjoint which preserves all colimits.

Lastly, notice that by Theorem \ref{thm:BigMonadicity}(1), $u_*\colon\LE(d,k)\to \LE(ud,k)$ has a left adjoint if $f_K$ preserves colimits of $\kappa$-chains and $\K(K,c)$ has $\kappa_\leq$-colimits.
\end{proof}

\begin{theorem}\label{thm:BigMonadicityB}
    Fix an infinite cardinal $\kappa$. Let $\B = (\C,\ F,G,\ F',G',\ \alpha, \beta)$ be a bialgebraic datum in $\CAT$. Assume the following:
    \begin{enumerate}
        \item Functors $G$ and $G'$ have $F$- and $F'$-relative left adjoints $L_1$ and $L_2$, respectively.
        \item The category $\C$ has colimits of size at most $\kappa$.
        \item The functor $F$ preserves the colimiting cocones of $\kappa$-chains and $F'$ their joint epicness.
    \end{enumerate}
    Then $U\colon \BialgB\to \C$ is monadic and $\BialgB$ has colimits of size at most $\kappa$. Furthermore, if $K\xleftarrow{k} I\xrightarrow{d}\BialgB$ is a span of functors, where the left Kan extensions of $Ud$ and $L_1Ud$ along $k$ exist, then the left Kan extension of $d$ along $k$ exists.
\end{theorem}
\begin{proof}
    For monadicity, it is sufficient to show that the forgetful functors $U'\colon \Ins(F,G)\to \C$ and $E\colon \BialgB\to \Ins(F,G)$ both have a left adjoint (Corollary \ref{cor:GeneralMonadicity}). Without loss of generality, we may assume that $G$ and $G'$ are the identity functors on $\C$ as in Remark \ref{rmk:AdjointsInducingIsosBetweenBialgebras} (the relative left adjoints inherit the corresponding properties from the arity functors). Since $F$ preserves $\kappa$-colimits and $\C$ has colimits of size at most $\kappa$, $U'\colon \Ins(F,\id_\C)\to \C$ is monadic and $\Ins(F,\id_\C)$ has colimits of size at most $\kappa$ (Theorem \ref{thm:BigMonadicity} (1) and Lemma \ref{lem:KanExtensionInserter}). Note also that the functor $U'\colon \Ins(F,\id_\C)\to \C$ strictly creates $\kappa$-colimits, since $F$ preserves such colimits (Theorem \ref{thm:pie_kan_creation} (2)).

    Next we show that the forgetful functor $E\colon \Eq(\alpha,\beta)\to \Ins(F,\id_\C)$ has a left adjoint. Notice that $F'U'$ preserves the joint epicness of colimiting cocones of $\kappa$-chains and $\Ins(F,\id_\C)$ has coequalizers and $\kappa_\leq$-colimits. By Theorem \ref{thm:BigMonadicity}(2), $E$ has a left adjoint.

    Assume then that $K\xleftarrow{k} I\xrightarrow{d} \BialgB$ is a span of functors and assume that $Ud$ and $FUd$ have left Kan extensions along $k$. We show that $d$ has a left Kan extension along $k$. By Lemma \ref{lem:KanExtensionInserter} we have that the left Kan extension of $Ed$ exists along $k$ and since the right adjoint $E$ is fully faithful, the left Kan extension of $d$ along $k$ exists.
\end{proof}

\begin{corollary}\label{cor:BigMonadicityModel}
    Let $\kappa$ and $\lambda$ be infinite cardinals. Let $(\Sigma,\sigma)$ be a signature pair with at most $\lambda$-many symbols with an algebraic $\sigma$-theory $T$. Let $M$ be a categorical $\Sigma$-model. Assume the following:
    \begin{enumerate}
        \item Each category $M_c$ has colimits of size at most $\max(\lambda,\kappa)$.
        \item Each functor $M(B)$ preserves $\kappa$-colimits componentwise.
    \end{enumerate}
    Then the forgetful functor $M^T\to M_0$ is monadic and $M^T$ has colimits of size at most $\max(\lambda,\kappa)$. Moreover, for any span $K\xleftarrow{k} I\xrightarrow{d} M^T$ of functors, the left Kan extension of $d$ along $k$ exists if $Ud$ and $M(t_1)Ud$ have left Kan extensions along $k$ for every function arity term $t_1$.
\end{corollary}
\begin{proof}
    The coarity functors of the induced bialgebraic datum $\B_{M,T}$ have left adjoints, because the coarity supports $M(t_2)$ are projections ($T$ algebraic) and the category $M_0$ has coproducts of size up to $\lambda$. The claim then follows from Theorem \ref{thm:BigMonadicityB} and the fact that $\kappa$-colimits are sifted colimits and the functors $M(B)$ preserving sifted colimits componentwise implies that the support functors preserve them (Corollary~\ref{cor:liftingSifted}).
\end{proof}
\begin{corollary}
    Let $\C$ be a cartesian closed category with countable colimits. Then for any single-sorted theory $T$ of universal algebra with a countable signature, the category $\C^T$ has countable colimits and the forgetful functor $U\colon \C^T\to \C$ is monadic and $U$ strictly creates countable sifted colimits.
\end{corollary}
\begin{proof}
    The monadicity and existence of countable colimits follow from Corollary~\ref{cor:BigMonadicityModel}. The strict creation of countable sifted colimits follows from Corollary \ref{cor:liftingSifted}.
\end{proof}

\subsection{Accessibility and Comonadicity}\label{subsec:comonadicity}
Consider a regular cardinal $\kappa$. A $\kappa$-accessible category is one that has a small set $A$ of objects with $\C$ having and $\C(a,-)$ preserving $\kappa$-directed colimits, for each $a\in A$, and for each object $c$ in $\C$ there is a diagram $D\colon I\to \C$ so that $c\cong \colim D$, where $I$ is a $\kappa$-directed poset and $D(\Obj(I))\subset A$ \cite{LPC}. If one such $\kappa$ exists, $\C$ is called an accessible category. A functor whose domain and codomain are $\kappa$-accessible categories and which preserves $\kappa$-directed colimits for some regular cardinal $\kappa$ is called an accessible functor. The $2$-category $\ACC$ of accessible categories, accessible functors and natural transformations is closed under PIE limits \cite{MakkaiPare1989AccessibleCategories, LPC} and an accessible category that is complete, cocomplete or locally presentable is all three \cite[Corollary 2.47]{LPC}. We explicitly state the following known theorem:
\begin{theorem}[Accessible Adjoint Functor Theorem]\label{AAFT}
    Let $F\colon \C\to \D$ be a functor with $\C$ locally presentable. Then the following hold:
    \begin{enumerate}
        \item If $F$ preserves colimits and $\D$ is locally small, then $F$ is a left adjoint.
        \item If $F$ is continuous and accessible, then $F$ is a right adjoint.
    \end{enumerate}
\end{theorem}
\begin{proof}
The first claim follows from the dual of the Special Adjoint Functor Theorem \cite{MacLane1998} and the fact that a locally presentable category has a small generating set and is co-well-powered \cite{LPC}.

The second claim follows from the fact that the $2$-category $\ACC$ is closed under comma objects in $\CAT$ (\cite{MakkaiPare1989AccessibleCategories}) and that an accessible category is complete if and only if it is cocomplete. The right adjointness then follows from the cocompleteness of the categories $d\downarrow F$.
\end{proof}

\begin{theorem}\label{thm:LPandComonadicity}
    Let $\B = (\C,\ F,G,\ F',G',\ \alpha,\beta)$ be a bialgebraic datum in $\CAT$ with $\C$ locally presentable and $F,F',G,G'$ accessible. Then $U_\B\colon \BialgB\to \C$ is accessible and the following hold:
    \begin{enumerate}
        \item If $G$ is continuous and $G'$ preserves joint monicness of limiting cones, then $\BialgB$ is locally presentable and the forgetful functor $U_\B\colon \BialgB\to \C$ is monadic.
        \item If $F$ preserves colimiting cocones and $F'$ preserves their joint epicness, then $\BialgB$ is locally presentable and $U_\B$ is comonadic.
    \end{enumerate}
\end{theorem}

\begin{proof}
    Since $\ACC$ is closed under PIE limits in $\CAT$ by \cite[2.72 Theorem, 2.76 Lemma]{LPC}, it follows that $U$ is accessible. As the cases are similar, we prove the comonadicity case.  As $U$ strictly creates all colimits, $\BialgB$ is cocomplete and hence locally presentable. Since $U$ is a cocontinuous accessible functor, it is a left adjoint by Theorem \ref{AAFT}. Thus $U$ is comonadic by Corollary~\ref{cor:GeneralMonadicity}.
\end{proof}

\section{Factorization Systems, Regularity, and Exactness}\label{sec:ofs}

To establish how regularity and exactness lift along the construction $M\mapsto M^T$ in $\CAT$, we first study how orthogonal factorization systems lift along PIE limits in $\CAT$.

\subsection{Orthogonal Factorization Systems}
\begin{definition}[Orthogonal Factorization System]\label{def:1cat_ofs}
   Let $\C$ be a category. Two morphisms $e$ and $m$ in $\C$ are said to be \textbf{orthogonal}, denoted $e \perp m$, if for every commutative square
   \[
\begin{tikzcd}
a \arrow[d, "e"'] \arrow[r, "f"]                              & c \arrow[d, "m"] \\
b \arrow[r, "g"'] \arrow[ru, "\exists!h" description, dotted] & d
\end{tikzcd}
   \] there exists a unique morphism $h\colon b\to c$, a \textbf{diagonal filler}, making the diagram commute. We write $e\perp_! m$ if at most one diagonal filler $h$ exists; we do not require its existence. We extend the notation $A\perp B$ to include classes of morphisms $A$ and $B$ in $\C$. We denote $A^\perp\coloneqq \{m\in \Mor(\C)\mid a\perp m\text{ for all $a\in A$}\}$, define ${}^\perp\! A$ dually, and call these the \textbf{right} and \textbf{left orthogonal classes} of $A$, respectively.

   An \textbf{orthogonal factorization system (OFS)} on $\C$ is a pair $(\mathcal{E}, \mathcal{M})$ of classes of morphisms in $\C$ such that:
   \begin{enumerate}
       \item \textbf{Factorization:} Every morphism $f$ in $\C$ admits a factorization $f = m \circ e$ with $e \in \mathcal{E}$ and $m \in \mathcal{M}$. Such a factorization is called an $(\mathcal{E},\mathcal{M})$-factorization.
       \item \textbf{Orthogonality:} $\mathcal{M} = \mathcal{E}^\perp$ and $\mathcal{E} = {}^\perp\! \mathcal{M}$.
   \end{enumerate}
   An OFS $(\mathcal{E},\mathcal{M})$ is called \textbf{stable} if $\mathcal{E}$ is stable under pullbacks.
\end{definition}

\begin{remark}
   An orthogonal factorization system can be equivalently described as a pair $(\mathcal{E}, \mathcal{M})$ of classes of morphisms of $\C$ such that every morphism in $\C$ admits an $(\mathcal{E}, \mathcal{M})$-factorization, $\mathcal{E} \perp \mathcal{M}$ and both $\E$ and $\M$ are closed under retracts in the arrow category $\C^{\to}$ \cite[Lemma~11.2.3]{Riehl2014CategoricalHomotopyTheory}. In practice, these closure properties are useful when verifying that a lifted pair of classes forms an orthogonal factorization system.
\end{remark}

We now show how these factorization systems strictly lift along PIE limits in $\CAT$.

\begin{theorem}[Lifting OFS Along PIE Limits]\label{thm:pie_ofs_creation}
   Let $\C_i$ and $\C$ be categories equipped with orthogonal factorization systems $(\E_i',\M_i')$ and $(\E',\M')$, respectively, for each $i\in I$. Let $F,G\colon \C\to \D$ be functors. Then the following hold.
   \begin{enumerate}
       \item $(\prod_i \mathcal{E}_i',\prod_i \mathcal{M}_i')$ is the unique OFS on $\prod_i \C_i$ preserved by the projections.
       \item If $F(\E')\perp G(\M')$, then $(U^{-1}(\mathcal{E}'),U^{-1}(\mathcal{M}'))$ is the unique OFS on $\Ins(F,G)$ preserved by the forgetful functor $U\colon \Ins(F,G)\to \C$.
       \item Let $\eta,\theta\colon F\Rightarrow G$ be natural transformations. If $F(\E')\perp_! G(\M')$, then the equifier $E\colon \Eq(\eta,\theta)\to \C$ has a unique OFS $(E^{-1}(\mathcal{E}'), E^{-1}(\mathcal{M}'))$ that is preserved by $E$.
   \end{enumerate}
\end{theorem}
\begin{proof}
Note in general that if $\M_1\subset \M_2$ and $\E_1\subset \E_2$ are classes of morphisms of some category, where $(\E_i,\M_i)$ is an OFS for $i = 1,2$, then $\M_1 = \M_2$ and $\E_1 = \E_2$. Hence the uniqueness claims of the statement hold.
\begin{enumerate}
   \item Clearly, $(\prod_i \mathcal{E}_i',\prod_i \mathcal{M}_i')$ is an orthogonal factorization system on $\prod_i \C_i$.

   \item Denote $\mathcal{E} \coloneqq U^{-1}(\mathcal{E}')$ and $\mathcal{M} \coloneqq U^{-1}(\mathcal{M}')$ and assume that $F(\E')\perp G(\M')$. Clearly, $\mathcal{E}$ and $\mathcal{M}$ contain all isomorphisms and are closed under composition and retracts, so it suffices to show that each morphism in $\Ins(F,G)$ has an $(\mathcal{E},\mathcal{M})$-factorization and $\mathcal{E}\perp \mathcal{M}$.

   Let $f\colon (x,\phi_x\colon Fx\to Gx)\to (z,\phi_z\colon Fz\to Gz)$ be a morphism in $\Ins(F,G)$. In $\C$, the morphism $f$ factors as $x\xrightarrow{e}y\xrightarrow{m}z$ with $e\in \mathcal{E}'$ and $m\in \mathcal{M}'$. Using the orthogonality $F(e)\perp G(m)$, we define the morphism $\phi_y\colon Fy\to Gy$ as the unique morphism making the diagram
   \[
   \begin{tikzcd}[column sep = 2 cm]
F(x) \arrow[d, "\phi_x"'] \arrow[r, "F(e)\in F(\E')"] & F(y) \arrow[d, "\exists!\phi_y" description] \arrow[r, "F(m)"] & F(z) \arrow[d, "\phi_z"] \\
G(x) \arrow[r, "G(e)"']                              & G(y) \arrow[r, "G(m)\in G(\M')"']                                   & G(z)
\end{tikzcd}
   \]
   commute. Thus we have a factorization $f = me$ in $\Ins(F,G)$, where $m\in \mathcal{M}$ and $e\in \mathcal{E}$. Finally, we show $\mathcal{E}\perp \mathcal{M}$. Consider the following lifting problem:
   \[
   \begin{tikzcd}
(x,\phi_x) \arrow[d, "e\in \mathcal{E}"'] \arrow[r, "f"] & (y,\phi_y) \arrow[d, "m\in \mathcal{M}"] \\
(v,\phi_v) \arrow[r, "g"']                      & (w,\phi_w)
\end{tikzcd}
   \]
   in $\Ins(F,G)$. There is a unique solution $\theta\colon v\to y$ in the underlying category $\C$, so the lifting problem has at most one solution. It suffices to show that $\theta$ is a morphism in $\Ins(F,G)$. We show that the diagram
   \[
   \begin{tikzcd}
F(v) \arrow[d, "\phi_v" description] \arrow[r, "F(\theta)"] & F(y) \arrow[d, "\phi_y" description] \\
G(v) \arrow[r, "G(\theta)"']                                  & G(y)
\end{tikzcd}
   \]
   commutes. The orthogonality $F(e)\perp G(m)$ induces a unique morphism $\phi\colon Fv\to Gy$ with $\phi F(e) = G(f)\phi_x$ and $G(m)\phi = \phi_w F(g)$; see the diagram
   \[
\begin{tikzcd}[column sep = 2 cm]
Fx \arrow[d, "\phi_x" description] \arrow[r, "F(e)" description] \arrow[rr, "Ff", bend left] & Fv \arrow[r, "F(\theta)" description] \arrow[d, "\phi_v" description] \arrow[rd, "\exists!\phi" description] \arrow[rr, "F(g)", bend left] & Fy \arrow[r, "F(m)" description] \arrow[d, "\phi_y" description] & Fw \arrow[d, "\phi_w" description] \\
Gx \arrow[r, "G(e)" description] \arrow[rr, "G(f)"', bend right]                                & Gv \arrow[r, "G(\theta)" description] \arrow[rr, "G(g)"', bend right]                                                                         & Gy \arrow[r, "G(m)" description]                                    & Gw
\end{tikzcd}
\]
   As both morphisms $G(\theta)\phi_v$ and $\phi_yF(\theta)$ satisfy the defining condition of $\phi$, it follows that $G(\theta)\phi_v = \phi_y F(\theta)$. Thus $\theta\colon (v,\phi_v)\to (y,\phi_y)$ is a morphism in $\Ins(F,G)$.

   \item Assume $F(\E')\perp_! G(\M')$ and write $\mathcal{E}\coloneqq E^{-1}(\mathcal{E}')$ and $\mathcal{M} \coloneqq E^{-1}(\mathcal{M}')$. We show that $(\mathcal{E},\mathcal{M})$ is an OFS on $\Eq(\eta,\theta)$. Let $f\colon x\to z$ be a morphism in $\Eq(\eta,\theta)$ and consider a factorization $f\colon x\xrightarrow{e}y\xrightarrow{m}z$, where $e\in \mathcal{E}'$ and $m\in \mathcal{M}'$. By the full faithfulness of $E$, it suffices to check that $\eta_y = \theta_y$ to show that $(\mathcal{E},\mathcal{M})$ is an orthogonal factorization system on $\Eq(\eta,\theta)$.
   Consider the following diagram:
   \[
\begin{tikzcd}[column sep = 2cm]
F(x) \arrow[d, "\eta_x = \theta_x" description] \arrow[r, "F(e)\in F(\mathcal{E}')"] & F(y) \arrow[d, "\eta_y"', shift right] \arrow[d, "\theta_y", shift left] \arrow[r, "F(m)"] & F(z) \arrow[d, "\eta_z = \theta_z" description] \\
G(x) \arrow[r, "G(e)"']                                                 & G(y) \arrow[r, "G(m)\in G(\mathcal{M}')"']                                                              & G(z)
\end{tikzcd}
\]
   As $F(e)\perp_! G(m)$, the diagonal filler is unique. Since both $\eta_y$ and $\theta_y$ make both squares commute, $\eta_y = \theta_y$. This shows that $y$ is an object of $\Eq(\eta,\theta)$. \qedhere
\end{enumerate}
\end{proof}

\subsection{Regularity and Exactness}

\begin{definition}
   Let $\C$ be a category. We say that $\C$ is a \textbf{regular} category if $\C$ has finite limits and a stable OFS $(\E,\M)$, where $\M$ is the class of monos in $\C$. A regular category $\C$ is called \textbf{exact} if each internal equivalence relation $R\rightrightarrows c$ in $\C$ is the kernel pair of some morphism in $\C$.\footnote{A jointly monomorphic pair of morphisms $R\rightrightarrows c$ in a category $\C$ is called an equivalence relation if the induced subset of $\hom(x,c)\times \hom(x,c)$ is an equivalence relation for each object $x$ in $\C$.}
\end{definition}

In a category $\C$ with pullbacks, we have that $\mathcal{E} =\! ^\perp\! \mathcal{M}$ and $\mathcal{M} = \mathcal{E}^\perp$ where $\mathcal{M}$ is the class of monomorphisms and $\mathcal{E}$ is the class of extremal epimorphisms in $\C$ (morphisms in $\E^\bot$ lift against their kernel pairs making them monic). If $\C$ is a regular category, then, by \cite[Proposition 2.2(b)]{Kelly1991}, $\E$ is the class of regular epis and from this one obtains the existence of coequalizers of kernel pairs. The $(\E,\M)$ factorization of a morphism $f$ is then obtained via the coequalizer of the kernel pair of $f$.

\begin{lemma}\label{lem:reflecting_regularity_exactness}
    Let $U\colon \D\to \C$ be a pullback preserving functor reflecting monics with $\D$ having finite limits. Let $\M = U^{-1}(\M')$ and $\E = U^{-1}(\E')$, where $\E'$ and $\M'$ are the classes of extremal epimorphisms and monomorphisms of $\C$, respectively. Then the following hold:
    \begin{enumerate}
        \item If $(\E,\M)$ is an OFS on $\D$ and $\E'$ is stable under pullbacks, then $\D$ is regular.
        \item If $\D$ has coequalizers of kernel pairs, $U$ is a regular epi preserving conservative functor, and the category $\C$ is regular, then $\D$ is regular.
        \item If $\D$ has coequalizers of equivalence relations preserved by the functor $U$, $U$ is conservative, and the category $\C$ is exact, then $\D$ is exact.
    \end{enumerate}
\end{lemma}
\begin{proof} Note that as $U$ preserves pullbacks, hence monics, and reflects monics, we obtain that $\M$ is the class of monomorphisms in $\D$.
    \begin{enumerate}
        \item Assume $(\E,\M)$ is an OFS on $\D$ and $\E'$ is stable. As $\D$ has finite limits, it suffices that $\E$ is pullback stable. Since $\E'$ is closed under pullbacks, $U$ preserves pullbacks and $\E = U^{-1}(\E')$, it follows that $\E$ is closed under pullbacks. Thus $\D$ is regular.
        \item Assume $\D$ has coequalizers of kernel pairs, $U$ is conservative and preserves regular epis, and $\C$ is regular. We show that $\D$ is regular. By part 1, it suffices to show that $(\E,\M)$ factorizations exist and $\E$ is the class of extremal epis in $\D$. Consider a morphism $f\colon d\to d'$ in $\D$ and consider its canonical factorization $f = me$, where $e$ is the coequalizer of the kernel pair of $f$. As $U$ preserves pullbacks and regular epis, $Uf = Um Ue$ is the canonical factorization in $\C$ and thus $Um$ is monic by regularity of $\C$. Thus $f$ has an $(\E,\M)$ factorization. As each extremal epi in $\D$ is regular and $U$ preserves regular epis, we have all extremal epis of $\D$ are in $\E$. Conversely, if a morphism $f$ is mapped to an extremal epi by $U$, then by conservativity of $U$, $f$ is an extremal epi and hence $\E$ is the class of extremal epis.
        \item Assume $\D$ has and $U$ preserves coequalizers of equivalence relations. Assume that $U$ is conservative and $\C$ is exact. We show that $\D$ is exact. By part $2$, $\D$ is regular. Consider then an equivalence relation $r_1,r_2\colon R\to x$ in $\D$, and consider the comparison $t\colon R\to S$, where $S$ is constructed as the kernel pair of the coequalizer $q$ of $r_1$ and $r_2$. Since $r_1,r_2$ is jointly monic, $t$ is a monomorphism. As $U$ preserves monomorphisms, $Ur_1,Ur_2$ is a relation in $\C$. The pullback preservation of $U$ shows that $Ur_1,Ur_2$ is an equivalence relation in $\C$ and thus the kernel pair of their coequalizer $Uq$. Thus $Ut$ is an isomorphism and by conservativity of $U$, $t$ is an isomorphism. Hence $r_1,r_2\colon R\to x$ is a kernel pair in $\D$. \qedhere
    \end{enumerate}

\end{proof}

Proposition 3.7 in \cite{Wissmann2022Minimality} shows how a monad $T\colon \C\to \C$ preserving the left class of an orthogonal factorization system $(\mathcal{E},\mathcal{M})$ on $\C$ induces an orthogonal factorization system on the Eilenberg--Moore category $\C^T$, generalizing the original result of Linton \cite{Linton1969}. We further generalize these results by lifting orthogonal factorization systems to arbitrary categories of bialgebras in the following theorem:

\begin{theorem}\label{thm:liftingregularexact}
    Let $\B = (\C,\ F,G,\ F',G',\ \alpha,\beta)$ be a bialgebraic datum in $\CAT$. Assume that $\C$ is equipped with an orthogonal factorization system $(\E',\M')$. Denote $\M = U_\B^{-1}(\M')$ and $\E = U_\B^{-1}(\E')$. Then the following conditions hold:
    \begin{enumerate}
        \item If $F(\E')\perp G(\M')$ and $F'(\E')\perp_! G'(\M')$, then the pair $(\E,\M)$ is the unique OFS on $\BialgB$ preserved by $U_\B\colon \BialgB\to \C$.
        \item Assume $\C$ is regular, $\M'$ is the class of monomorphisms in $\C$, and $\BialgB$ has a terminal object. Assume that $F$ preserves regular epis, $F'(\E')\perp_! G'(\M')$, $G$ preserves pullbacks, and $G'$ the joint monicness of pullback projections. Then $\BialgB$ is regular with $U$ strictly creating pullbacks and preserving regular epimorphisms.
        \item If $\C$ is exact, $\BialgB$ has a terminal object, $F$ preserves coequalizers of kernel pairs, $F'$ maps regular epis to epis, $G$ preserves pullbacks and $G'$ the joint monicness of pullback projections, then $\BialgB$ is exact with $U$ strictly creating coequalizers of kernel pairs.
    \end{enumerate}
\end{theorem}
\begin{proof}
    The first part follows from Theorem \ref{thm:pie_ofs_creation}. We then prove the second part. By Corollary \ref{cor:strictCreationToBialgebraicDatumObjectsKan}, $U$ strictly creates pullbacks. By the first part then $(\E,\M)$ is an OFS on $\BialgB$. Thus $\BialgB$ is regular with $U$ preserving regular epimorphisms by Lemma~\ref{lem:reflecting_regularity_exactness}(1).

    We show the third part. Assume that $\C$ is exact, $\BialgB$ has a terminal object, $F$ preserves coequalizers of kernel pairs, $F'$ maps regular epis to epis, $G$ preserves pullbacks and $G'$ the joint monicness of pullback projections. Part 2 shows that $\BialgB$ is regular. Since $U$ preserves pullbacks, it maps equivalence relations to kernel pairs. Since these kernel pairs have coequalizers which are preserved by $F$ and whose quotient initiality is preserved by $F'$, it follows from Corollary \ref{cor:strictCreationToBialgebraicDatumObjectsKan} that $U$ strictly creates this colimit. Lemma~\ref{lem:reflecting_regularity_exactness}(3) shows that $\BialgB$ is exact.
\end{proof}

\begin{corollary}
    Let $M$ be a $\Sigma$-model in $\CAT$. Assume each category $M_c$ is equipped with an OFS $(\E_c,\M_c)$. Let $T$ be a $\sigma$-theory and denote $(\E,\M) \coloneqq (U^{-1}(\prod_s \E_{c_s}),U^{-1}(\prod_s \M_{c_s}))$, where $U\colon M^T\to M_0$ is the forgetful functor. Then the following hold:
    \begin{enumerate}
        \item If $T$ is (co)algebraic and each functor $M(B)$ preserves the (right) left classes componentwise, then $(\E,\M)$ is the unique OFS on $M^T$ preserved by $U$.
        \item If each functor $M(B)$ preserves both classes componentwise, then $(\E,\M)$ is the unique OFS on $M^T$ that is preserved by $U$.
    \end{enumerate}
\end{corollary}
\begin{proof}
    Follows immediately from Theorem~\ref{thm:liftingregularexact}(1).
\end{proof}

\begin{corollary}\label{cor:ExactnessReflectionPrinciple}
    Let $\B_i = (\C_i,\ F_i,G_i,\ F_i',G_i',\ \alpha_i,\beta_i)$ be a bialgebraic datum in $\CAT$ for $i = 1,2$. Consider a commutative diagram
    $$
    \begin{tikzcd}
\Bialg(\B_1) \arrow[d, "U_1"'] \arrow[r, "\bar{F}"] & \Bialg(\B_2) \arrow[d, "U_2"] \\
\C_1 \arrow[r, "F"']                                & \C_2
\end{tikzcd}
    $$
    of functors, where $U_i = U_{\B_i}$ for $i = 1,2$. Assume the following:
    \begin{enumerate}
        \item The category $\C_1$ has and $F$ preserves pullbacks and reflexive coequalizers.
        \item The functor $F$ is conservative.
        \item The function arity functors $F_1, F_2$ preserve reflexive coequalizers and the equational arity functors $F_1',F_2'$ map regular epis to epis.
        \item The function coarity functors $G_1$ and $G_2$ preserve pullbacks and the equational coarity functors $G_1',G_2'$ preserve the joint monicness of pullback projections.
    \end{enumerate}
    Then $\Bialg(\B_1)$ has and $\bar{F}$ preserves pullbacks and reflexive coequalizers, and $\bar{F}$ is conservative. Moreover, if $\Bialg(\B_2)$ is exact and $\Bialg(\B_1)$ has a terminal object, then $\Bialg(\B_1)$ is exact.
\end{corollary}
\begin{proof}
    As the forgetful functors $U_i$ strictly create pullbacks and reflexive coequalizers for $i = 1,2$ (Corollary \ref{cor:strictCreationToBialgebraicDatumObjectsKan}), $\Bialg(\B_1)$ has these (co)limits and the functor $\bar{F}$ preserves them. Since $F,U_1$ and $U_2$ are conservative, so is $\bar{F}$. By Lemma~\ref{lem:reflecting_regularity_exactness}(3), $\bar{F}$ reflects exactness when its domain has a terminal object.
\end{proof}
\begin{corollary}\label{cor:exactnessLiftingAlongAlgebraicExtensionsOfTheories}
    Let $\sigma_1\subset\sigma_2$ be $1$-signatures relative to $\Sigma$ with the same set of sorts. Let $T_1\subset T_2$ be algebraic $\sigma_1$- and $\sigma_2$-theories, respectively, with $M$ a categorical $\Sigma$-model. Assume that each category $M_c$ has finite limits and reflexive coequalizers and each functor $M(B)$ preserves reflexive coequalizers componentwise. If $M^{T_1}$ is exact, then so is $M^{T_2}$.
\end{corollary}
\begin{proof}
    Follows directly by considering the forgetful functor $U\colon M^{T_2}\to M^{T_1}$ and applying Corollary~\ref{cor:ExactnessReflectionPrinciple}.
\end{proof}

\subsection{Examples of Exact Categories via Theory Extension}

Consider a commutative ring $R$. The ring $R$ is said to be von Neumann regular, or just regular, if each element $r\in R$ has a pseudo-inverse $s\in R$, that is, $rsr = r$ and $srs = s$. The pseudo-inverse is necessarily unique, permitting a description of commutative regular rings with a theory of universal algebra. For each maximal ideal $\alpha$ of $R$, one can consider the localization $R_\alpha\coloneqq (R\setminus \alpha)^{-1}R$.
\begin{lemma}\label{lem:localizations}
    Let $T$ be a bialgebraic theory, with $S$ as the set of sorts, over the $2$-signature for symmetric monoidal categories. Then the localizations at maximal ideals $\alpha$ determine a jointly conservative family
    $$
    (L_\alpha^T\colon \Mod_R^T\to\Mod_{R_\alpha}^T)_\alpha
    $$
    of functors preserving sifted colimits. In the case $R$ is a regular ring, the localization $L^T_\alpha$ preserves finite cosifted limits and $R_\alpha$ is a field for each maximal ideal $\alpha$ of $R$.
\end{lemma}
\begin{proof}
    Consider a maximal ideal $\alpha$ of $R$ and the localization $L_\alpha\colon \Mod_R\to \Mod_{R_\alpha}, M\mapsto M\otimes_{R} R_\alpha$. The functor $L_\alpha$ is exact, symmetric monoidal, and jointly conservative over all maximal ideals $\alpha$ by \cite[Chapter 3]{atiyah_introduction_1969}. Denote the induced functor of algebras by $L_\alpha^T\colon \Mod_R^T\to \Mod_{R_\alpha}^T$ (Theorem \ref{thm:induced_morphism_between_algebras}). The joint conservativity claim is clear. Since the monoidal products preserve colimits componentwise, the forgetful functor $\Mod_R^T\to \Mod_R^S$ strictly creates all sifted colimits (Corollary \ref{cor:liftingSifted}). Thus $L_\alpha^T$ preserves all sifted colimits.

    In the case $R$ is a regular ring, we have by \cite[Chapter~2, Exercise~27; Chapter~3, Exercise~10]{atiyah_introduction_1969} that the localization $R_\alpha$ is a field and the tensor $\otimes_R$ is left-exact componentwise. Thus $\Mod_R^T\to \Mod_R^S$ (similarly for the codomain of $L_\alpha^T$) strictly creates finite cosifted limits and hence $L_\alpha^T$ preserves them.
\end{proof}
Consider a symmetric monoidal category $A$ and denote by $\mathbf{Comon}^{\mathrm{coc}}_A$ the category of internal cocommutative comonoids. The category of internal cocommutative Hopf algebras $\Hopf_A^{\mathrm{coc}}$ in $A$ is defined as the category $\mathbf{Grp}(\mathbf{Comon}^{\mathrm{coc}}_A)$ of internal groups within $\mathbf{Comon}^{\mathrm{coc}}_A$. In the case where $A$ is the symmetric monoidal category of $k$-vector spaces or its dual, it has been shown that the category $\Hopf^{\mathrm{coc}}_A$ is exact (even semi-abelian) \cite{GRAN20194171, forsman2026semiabeliannessaffinegroupschemes}. The following corollary generalizes some exactness results found in \cite{bevilacquaCoalgebraicModelsOmega2026, granHopfBracesSemiabelian2026} by showing that $(\mathbf{Comon}^{\mathrm{coc}}_A)^T$ is exact for all algebraic theories $T$ extending the theory of groups, when $A$ is either the symmetric monoidal category of $R$-modules or its dual for a commutative regular ring $R$:
\begin{corollary}\label{cor:algebraicExtensionLiftingExactness}
    Consider a commutative regular ring $R$. Let $T$ be a multi-sorted theory of universal algebra extending the theory of groups in each sort $s\in S$. Then $(\Comon_A^\coc)^T$ is exact for $A = \Mod_R, \Mod_R^\mathrm{op}$.
\end{corollary}
\begin{proof}
    Consider first the case $A = \Mod_R$ and consider the following functors over maximal ideals $\alpha$ from Lemma~\ref{lem:localizations}
    $$
    (\Comon_R^\coc)^T\xrightarrow{U} (\Hopf_R^\coc)^S\xrightarrow{L_\alpha^S} (\Hopf_{R_\alpha}^\coc)^S
    $$
    Since $R$ is a regular ring, the localizations preserve coreflexive equalizers and finite products, so they preserve finite limits. Thus $\Hopf_R^\coc$ is exact by Lemma \ref{lem:reflecting_regularity_exactness}(3) and the fact that $\Hopf_k^\coc$ is exact for all fields $k$ by \cite{GRAN20194171}. Then by Corollary~\ref{cor:exactnessLiftingAlongAlgebraicExtensionsOfTheories}, $(\Comon_R^\coc)^T$ is exact.

    Now consider the case $A = \Mod_R^{\mathrm{op}}$. Note that $\Comon_A^\coc$ is the category of affine schemes over $R$. The exactness of affine group schemes over a field is shown in \cite{forsman2026semiabeliannessaffinegroupschemes}. Thus the exactness of $(\Comon_A^\coc)^T$ follows in the same manner as the previous case.
\end{proof}

Recently, there has been interest in cocommutative Hopf braces \cite{agoreCategoryHopfBraces2026, Angiono_2016}, which are skew braces within the category $\Comon_k^\coc$, and $\Omega$-Hopf algebras, which are internal $\Omega$-groups \cite{bevilacquaCoalgebraicModelsOmega2026} within $\Comon_k^\coc$. Corollary~\ref{cor:algebraicExtensionLiftingExactness} immediately shows their exactness.

\section{Monadicity and Local Presentability in \texorpdfstring{$2$}{2}-Categories}\label{sec:representable}

Using representable $2$-functors, we generalize the monadicity results of Section~\ref{sec:monadicity} to objects in PIE-complete $2$-categories with a choice of a bi-dense class of objects. For this section, we fix a $2$-category $\K$ with a class of objects $A$ and denote by $\A$ the full sub-$2$-category spanned by $A$.

\subsection{Bi-Density and Representable Properties}

For $2$-functors $F,G\colon \K\to \mathcal{L}$, we write $[F,G]_\Psd$ for the category of pseudonatural transformations $\eta \colon F\Rightarrow G$ and modifications $\Phi\colon \eta\Rrightarrow\theta$ (see Example~\ref{ex:modificationsAndTransformations}).

\begin{definition}
We say that the class of objects $A$ is \textbf{bi-dense} in $\K$ if the $2$-functor $c\mapsto \K(-,c)|_{\A^{\op}}\colon \K\to [\A^{\op},\CAT]$ is bi-fully faithful, that is, for all objects $a,b$ in $\K$, the functor
\[
\K(a,b)\to [\K(-,a)|_{\A^{\op}},\K(-,b)|_{\A^{\op}}]_\Psd,\quad f\mapsto f_*\coloneqq f*(-),
\]
is an equivalence of categories. This is a variation of the bicategorical notion of density introduced in \cite[Definition~6.2]{hofstraInnerAutoequivalencesGeneral2024}.
\end{definition}
The bicategorical Yoneda lemma \cite[(1.9)]{Street1980FibrationsBicategories} shows that the class of all objects of a $2$-category $\K$ is bi-dense. In the appendix, we show that, for a monoidal category $\V$ with an \textbf{absorbing initial object}, the linear $\V$-categories $\mathbf{1}$, $\mathbf{2}_x$ and $\mathbf{3}_{x,y}$, for objects $x$ and $y$ in $\V$, form a bi-dense class in $\V\text{-}\CAT$ (Theorem~\ref{thm:density_of_A_app}).

\begin{definition}[$A$ Properties]\label{def:representableProperties}
Consider a property $P$ of categories such as completeness, cocompleteness and accessibility. We say that an object $c$ in $\K$ is $A$-$P$ ($A$-\textbf{complete}, for example) if the category $\K(a,c)$ has the property $P$ for each $a\in A$. If $P'$ is a property of functors, we say that a morphism $f\colon c\to d$ in $\K$ is $A$-$P'$ if each functor $f_a\colon \K(a,c)\to \K(a,d)$ has the property $P'$.
\end{definition}

\begin{definition}[$\A$-local presentability]\label{def:ALocalPresentability}
    Let $c$ be an object in $\K$. We say that $c$ is $\A$-\textbf{complete}, or has $\A$-\textbf{limits}, if $c$ is $A$-complete and has right Kan extensions along each morphism $k$ in $\A$. Similarly, we say that $c$ is $\A$-\textbf{cocomplete}, or has $\A$-colimits, if $c$ is $A$-cocomplete and has all left Kan extensions along morphisms in $\A$. We say that $c$ is (\textbf{left}) \textbf{right} $\A$-\textbf{locally presentable} if $c$ is $A$-accessible and $\A$-(co)complete. A morphism $f$ is said to be $\A$-(\textbf{co})\textbf{continuous} and to \textbf{preserve} $\A$-(\textbf{co})\textbf{limits} if $f$ is $A$-(co)continuous and preserves the (left) right Kan extensions along morphisms in $\A$.
\end{definition}

\subsection{Monadicity and Comonadicity in \texorpdfstring{$2$}{2}-Categories}

We consider a PIE-complete $2$-category $\K$ with a bi-dense class of objects $A$, where we denote by $\A$ the full sub-$2$-category of $\K$ spanned by $A$. A morphism being an $A$-right adjoint is not sufficient for it to have a left adjoint, but with the following additional assumption we obtain the left adjoint from the local data:

\begin{theorem}\label{thm:LiftingOfAdjoint}
     Let $g\colon d\to c$ be a morphism in $\K$. Assume that $d$ has and $g$ preserves all right Kan extensions along morphisms in $\A$. If $g$ is an $A$-right adjoint, then $g$ is a right adjoint.
\end{theorem}
\begin{proof}
    Denote by $L_a\colon \K(a,c)\to \K(a,d)$ the left adjoint of $g_a$ for each $a\in A$. We show that $(L_a)_{a\in A}$ defines a pseudonatural transformation $\K(-,c)\Rightarrow \K(-,d)\colon \A^{\op}\to \CAT$. Let $k\colon a\to b$ be a morphism in $\A$. We show that the diagram
\[
\begin{tikzcd}
{\K(b,c)} \arrow[d, "k^c"'] \arrow[r, "L_b"] & {\K(b,d)} \arrow[d, "k^d"] \\
{\K(a,c)} \arrow[r, "L_a"']                  & {\K(a,d)}
\end{tikzcd}
\]
commutes up to coherent isomorphism. Note that the family $(L_a)_a$ forms an oplax natural transformation via the canonical composite
\begin{equation}\label{cancomp:1}
L_ak^c\Rightarrow L_ak^cg_bL_b = L_ag_ak^dL_b\Rightarrow k^dL_b
\end{equation}
By \cite[Proposition~2.35]{houstonLinearLogicUnits2013}, it suffices to show that the composite \eqref{cancomp:1} is an isomorphism for $L$ to be a left adjoint of $g_*$. Note that $L_ak^c$ and $k^dL_b$ have the same right adjoint:
\begin{align*}
    (L_ak^c(x),y) &\cong (xk,gy)\tag{$L_a\dashv g_a$}\\
    &\cong (x, g_b\circ \Ran_k^d(y))\tag{$g$ preserves $\Ran_k$}\\
    &\cong (L_b(x),\Ran_k^d(y)) \tag{$L_b\dashv g_b$}\\
    &\cong (k^dL_b(x),y)    \tag{$k^d\dashv \Ran_k^d$}
\end{align*}
for each $x\colon b\to c$ and $y\colon a\to d$. Following this natural isomorphism shows that the canonical composite \eqref{cancomp:1} is an isomorphism. Thus $g_*\colon \K(-,d)|_{\A^{\op}}\Rightarrow\K(-,c)|_{\A^{\op}}$ is a right adjoint, and by the bi-density of $A$, $g$ is a right adjoint.
\end{proof}
\begin{corollary}\label{cor:2-categoricalAAFT}
    Let $g\colon d\to c$ be an $A$-accessible morphism in $\K$. If $d$ has and $g$ preserves $\A$-limits, then $g$ is a right adjoint. Moreover, if $d$ has and $g$ preserves $\A$-colimits, then $g$ is a left adjoint.
\end{corollary}
\begin{proof}
    Assume that $d$ has and $g$ preserves $\A$-limits. We show that $g$ is a right adjoint in $\K$. Since for every $a\in A$, $g_a\colon \K(a,d)\to \K(a,c)$ is accessible and continuous with $\K(a,d)$ being complete, we have that $g_a$ has a left adjoint (Theorem \ref{AAFT}). Since $A$ is bi-dense and $g$ preserves all right Kan extensions along morphisms in $\A$, we have that $g$ is a right adjoint by Theorem \ref{thm:LiftingOfAdjoint}. The left adjoint case follows similarly.
\end{proof}
\begin{theorem}\label{thm:2CategoricalMonadicityForBialgebras}
    Let $\B = (c,\ f,g,\ f',g',\ \alpha,\beta)$ be a bialgebraic datum in $\K$. Assume:
    \begin{enumerate}
        \item The morphisms $g$ and $g'$ are right adjoints.
        \item The object $c$ is $A$-cocomplete and has all right Kan extensions along morphisms in $\A$.
        \item The functors $f_a\coloneqq \K(a,f)\colon \K(a,c)\to \K(a,x)$ and $f_a'\coloneqq \K(a,f')\colon \K(a,c)\to \K(a,y)$ preserve $\kappa_a$-colimits, where $\kappa_a$ is an infinite cardinal, for each $a\in A$.
    \end{enumerate}
    Then $u_\B\colon \BialgB\to c$ is monadic. Furthermore, if $K\xleftarrow{k}I\xrightarrow{d} \BialgB$ is a span in $\K$ so that $K\in A$, and the left Kan extensions of $ud$ and $fud$ along $k$ exist, then the left Kan extension of $d$ exists along $k$. In particular, $\BialgB$ is $A$-cocomplete.
\end{theorem}
\begin{proof}
    By Corollary \ref{cor:GeneralMonadicity}, for monadicity of $u$ it is sufficient to show $u$ is a right adjoint. Notice that $u_a\colon \K(a,\BialgB)\to \K(a,c)$ has a left adjoint for each $a\in A$ by Theorem~\ref{thm:BigMonadicityB}. As $c$ has and $u$ strictly creates $k$-limits, $\BialgB$ has $k$-limits and $u$ preserves them for each morphism $k$ in $\A$. Thus the local left adjoints to $u$ provide a left adjoint to $u$ by Theorem \ref{thm:LiftingOfAdjoint}. Thus $u$ is monadic. The last part follows from Lemma \ref{lem:KanExtensionInserter}.
\end{proof}

\begin{theorem}\label{thm:2CatLPComonadicity}
    Let $\B = (c,\ f,g,\ f',g',\ \alpha,\beta)$ be a bialgebraic datum in $\K$, where the morphisms $f,f',g,g'$ are $A$-accessible. Then the following hold:
    \begin{enumerate}
        \item If $c$ has and $g,g'$ preserve $\A$-limits, then $u_\B\colon \BialgB\to c$ is monadic, $A$-accessible and $\BialgB$ is right $\A$-locally presentable.
        \item If $c$ has and $f,f'$ preserve $\A$-colimits, then $\BialgB$ is left $\A$-locally presentable with $u$ being comonadic and $A$-accessible.
    \end{enumerate}
\end{theorem}
\begin{proof}
    As the cases are similar, we prove the latter comonadicity case. Assume that $c$ has and $f,f'$ preserve $\A$-colimits. By Corollary~\ref{cor:strictCreationToBialgebraicDatumObjectsKan}, we have that  $u\colon \BialgB\to c$ strictly creates $\A$-colimits and since $c$ has $\A$-colimits, $\BialgB$ has and $u$ preserves $\A$-colimits. As $\BialgB$ is $A$-cocomplete and $A$-accessible, it is $A$-locally presentable. Thus $\BialgB$ is left $\A$-locally presentable and $u$ is $\A$-cocontinuous. Since $u$ is $A$-accessible, it has a right adjoint (Corollary~\ref{cor:2-categoricalAAFT}) and hence it is comonadic by Corollary \ref{cor:GeneralMonadicity}.
\end{proof}

\appendix
\section{Bi-Density of Small \texorpdfstring{$\V$}{V}-Categories in \texorpdfstring{$\V\text{-}\CAT$}{V-CAT}}

Fix a monoidal category $\V$ with an \textbf{absorbing initial object} $0$, that is, an initial object that the tensor product preserves componentwise: $0\otimes z\cong 0\cong z\otimes 0$. We will construct a class $A$ of $\V$-categories and prove its bi-density. The proof is closely analogous to the one in \cite[Lemma~7.1]{hofstraInnerAutoequivalencesGeneral2024}, where a similar statement is shown for monoidal categories.

\begin{definition}
    Consider sets $E,G_0$ with a parallel pair of functions $E\rightrightarrows G_0$. We call the tuple $G = (E\rightrightarrows G_0)$ a \textbf{graph}. A $\V$-\textbf{graph} $\textbf{G}$ over $G_0$ is just a choice of hom-objects $\textbf{G}(a,b)$ in $\V$ for each $a,b\in G_0$. The elements of $G_0$ and $E$ are called vertices and edges, respectively. A graph $G$ is said to be a \textbf{forest} when the free category over $G$ is thin, that is to say, there is at most one path $x\to y$ in $G$ for any vertices $x$ and $y$ in $G$.
\end{definition}

\begin{definition}
    Assume $G$ is a forest. A \textbf{$\V$-labelling} of $G$ is a function $\ell\colon E\to\Obj(\V)$. For a path $p=p_n\cdots p_1$ we write $\ell(p)\coloneqq\ell(p_n)\otimes\cdots\otimes\ell(p_1)$, with $\ell(p)=I$ for the empty path. The labelling determines a $\V$-graph $G_\ell$ with vertices $G_0$ and hom-objects $G_\ell(a,b)=\ell(e)$ if there is an edge $e\colon a\to b$, and $G_\ell(a,b)=0$ otherwise.
\end{definition}

\begin{lemma}\label{lem:freeForest}
Let $G$ be a forest with a $\V$-labelling $\ell$. Then the $\V$-graph $G_\ell$ has a free $\V$-category $\overline{G}_\ell$, with objects $G_0$ and hom-objects
$$
\overline{G}_\ell(a,b)=\begin{cases}\ell(p), & \text{if } p\colon a\to b \text{ is a path},\\ 0, & \text{otherwise}.\end{cases}
$$
The composition $\overline{G}_\ell(b,c)\otimes \overline{G}_\ell(a,b)\to \overline{G}_\ell(a,c)$ is either the unique morphism from the initial object or the canonical isomorphism $\ell(q)\otimes \ell(p)\cong \ell(qp)$ induced by the monoidal structure, when paths $p\colon a\to b$ and $q\colon b\to c$ exist.
\end{lemma}
\begin{proof}
The construction of a free $\V$-category over $G_\ell$ given in \cite{wolffVcatVgraph1974} works as is, since the needed coproducts trivialize by the fact that the graph $G$ is a forest and $\V$ has an absorbing initial object.
\end{proof}

\begin{definition}\label{def:test_categories_A}
For $n\geq1$ and $x_1,\ldots,x_{n-1}\in\V$, the \textbf{linear $\V$-category} $\mathbf{n}_{x_1,\ldots,x_{n-1}}$ is $\overline{G}_\ell$ for the forest $0\to1\to\cdots\to n-1$ whose edge $(i-1)\to i$ is labelled by $x_i$ for each $i\leq n-1$.
\end{definition}

The $\V$-category $\mathbf{1}$ has a single object with hom-object $I$. The hom-objects of $\mathbf{2}_x$ and $\mathbf{3}_{x,y}$ above the diagonal are
$$
\mathbf{2}_x(0,1)=x,\qquad \mathbf{3}_{x,y}(0,1)=x,\quad \mathbf{3}_{x,y}(1,2)=y,\quad \mathbf{3}_{x,y}(0,2)=y\otimes x,
$$
and those below the diagonal are $0$.

Consider a monoidal category $\V$ with an absorbing initial object and consider a $\V$-category $\C$. Lemma \ref{lem:freeForest} allows us to fix notation as follows: A $\V$-functor
\begin{enumerate}
    \item $\mathbf{1}\to \C$ describes a choice of an object $c$ in $\C$, denoted by $\Delta_c\colon \textbf{1}\to\C$.
    \item $H\colon\mathbf{2}_x\to\C$ is equivalently described as a choice of objects $c,c'$ and a morphism $f\colon x\to\C(c,c')$. Motivated by this, we write $H=H_f$.
    \item $L\colon \mathbf{3}_{x,y}\to\C$ amounts to a choice of objects $c_0,c_1,c_2$ and morphisms $f\colon x\to\C(c_0,c_1)$, $g\colon y\to\C(c_1,c_2)$, with $L_{0,2}\colon y\otimes x\to \C(c_0, c_2)$ being $\mu_{c_0,c_1,c_2}\circ(g\otimes f)$, where $\mu$ denotes the composition of $\C$. We denote by $j_c\colon I\to \C(c,c)$ the identity of an object $c$ in $\C$.
\end{enumerate}

\begin{theorem}[Bi-Density of $\mathcal{A}$ in $\V\text{-}\CAT$]\label{thm:density_of_A_app}
Let $\V$ be a monoidal category with an absorbing initial object. Then the class
$$
A \coloneqq \{\mathbf{1}, \mathbf{2}_x, \mathbf{3}_{x,y} \mid x,y \in \Obj(\V)\}
$$
is bi-dense in the $2$-category $\K \coloneqq \V\text{-}\CAT$.
\end{theorem}

\begin{proof}
We only sketch the proof, since the argument closely follows the one given in \cite[Lemma~7.1]{hofstraInnerAutoequivalencesGeneral2024}. We denote by $\A$ the full sub-2-category $A$ spans in $\K$. Let $\C,\D$ be $\V$-categories. We show that
$$
\K(\C,\D)\to[\K(-,\C)|_{\A^{\op}},\K(-,\D)|_{\A^{\op}}]_\Psd, F\mapsto F_*,
$$
is fully faithful and essentially surjective on objects. A modification $\Theta\colon F_*\Rrightarrow G_*$, where $F,G\colon \C\to \D$ are $\V$-functors, is determined by the values $(\alpha_c\coloneqq (\Theta_\textbf{1})_{\Delta_c}\colon I\to \D(Fc,Gc))_c$ which provides the faithfulness, and the $\V$-naturality of $\alpha$ is encoded by $\Theta$ at components $\textbf{2}_x$, which shows fullness.

It remains to show essential surjectivity on objects. Consider a pseudonatural transformation
$$
\Phi\colon \K(-,\C)|_{\A^{\op}}\Rightarrow \K(-,\D)|_{\A^{\op}}\colon \A^{\op}\to \CAT,
$$
with the strong structure provided by isomorphisms $\gamma_f(H)\colon \Phi_\textbf{b}(H)f\Rightarrow \Phi_\textbf{a}(Hf)$ for morphisms $f\colon \textbf{a}\to \textbf{b}$ in $\A$ and $\V$-functors $H\colon \textbf{b}\to \C$. Note that for each $H\colon \textbf{a}\to \C$ and a vertex $i$ of \textbf{a}, we have an isomorphism
$$
\gamma_{v_i}(H)\colon \Phi_\textbf{a}(H)v_i\xrightarrow{\cong} \Phi_\textbf{1}(Hv_i)\eqqcolon \Phi'_\textbf{a}(H)(i),
$$
where $v_i\colon \textbf{1}\to \textbf{a}$ is the vertex inclusion choosing $i$ in \textbf{a}. The family $\Phi'$ extends uniquely to a $2$-natural transformation $\K(-,\C)|_{\A^{\op}}\Rightarrow \K(-,\D)|_{\A^{\op}}$, where the isomorphisms $\gamma_{v_{-}}(-)\colon \Phi\cong \Phi'$ form a modification. Therefore, we may assume that $\Phi$ itself is a strict natural transformation. The strict naturality of $\Phi$ means that the equations
\begin{equation*}
\Phi_\textbf{a}(FH) = \Phi_\textbf{b}(F)H,\qquad \Phi_\textbf{a}(\beta * H) = \Phi_\textbf{b}(\beta) * H,\qquad \Phi_\textbf{a}(F * \sigma) = \Phi_\textbf{b}(F) * \sigma
\end{equation*}
hold for all morphisms $H\colon \textbf{a}\to \textbf{b}$ in $\A$, $\V$-functors $F\colon \textbf{b}\to \C$, $\V$-natural transformations $\beta\colon F\Rightarrow F'\colon \textbf{b}\to \C$, and transformations $\sigma\colon H\Rightarrow H'\colon \textbf{a}\to \textbf{b}$ in $\A$.

We define $F\colon \C\to \D$ by setting $F(c)$ to be the object picked by $\Phi_\textbf{1}(\Delta_c)\colon \textbf{1}\to \D$. Consider then objects $c_1$ and $c_2$ in $\C$. Fix a morphism $f\colon x\to \C(c_1,c_2)$ in $\V$. The $\V$-functor $\Phi_{\textbf{2}_x}(H_f)\colon \textbf{2}_x\to \D$ corresponds to a morphism $\phi_x(f)\colon x\to \D(Fc_1,Fc_2)$, since $\Phi_{\textbf{2}_x}(H_f)v_{i-1} = \Phi_\textbf{1}(\Delta_{c_i})$ for $i = 1,2$. We then define $F_{c_1,c_2} = \phi_{\C(c_1,c_2)}(\id_{\C(c_1,c_2)})$.

We show that $F$ respects identities. Let $c$ be an object in $\C$. Consider the $\V$-functor $!\colon \textbf{2}_I\to \textbf{1}$ choosing the identity on the tensor unit $I$. As $H_{j_c} = \Delta_c\circ !$, we have
$$
H_{j_{Fc}} = \Delta_{Fc}\circ ! = \Phi_{\textbf{2}_I}(\Delta_c\circ !) = H_{F_{c,c}j_c}
$$
and thus $F_{c,c} j_c = j_{Fc}$ which shows that $F$ respects identities.

To see that $F$ respects composition, fix objects $c_0,c_1$ and $c_2$. One needs to show that
$$
F_{c_0,c_2}\circ \mu = \mu \circ (F_{c_1,c_2}\otimes F_{c_0,c_1}).
$$
Consider the $\V$-functor $L\colon \textbf{3}_{x,y}\to \C$ corresponding to the morphisms $\id_x\colon x = \C(c_0,c_1)$ and $\id_y\colon y = \C(c_1,c_2)$. Now $L_{0,2}$ is the composite $\mu_{c_0,c_1,c_2}$. Consider the morphism $d_{ij}^\textbf{a} = d_{ij}\colon \textbf{2}_{\textbf{a}(i,j)}\to \textbf{a}$ choosing the identity $\textbf{a}(i,j)\to \textbf{a}(i,j)$ for vertices $i$ and $j$ of $\textbf{a}\in A$. Now
$$
H_{\id_x} = L\circ d_{01}, \quad H_{\id_y} = L\circ d_{12},\quad  \text{and}\quad H_\mu = L\circ d_{02}.
$$
Due to the strict naturality of $\Phi$, we have
$$
\Phi_{\textbf{3}_{x,y}}(L)_{i,j} = \begin{cases}
    F_{c_0,c_1}, &\text{ if $ij = 01$,}\\
    F_{c_1,c_2}, &\text{ if $ij = 12$,}\\
    F_{c_0,c_2}\circ \mu, &\text{ if $ij = 02$.}
\end{cases}
$$
The $\V$-functoriality gives us that $\Phi_{\textbf{3}_{x,y}}(L)_{0,2} = \mu\circ (F_{c_1,c_2}\otimes F_{c_0,c_1})$. Thus $F$ is a $\V$-functor.

Lastly, we need to see that $\Phi = F_*$. Consider $\textbf{a}\in A$ and $\alpha\colon H_1\Rightarrow H_2\colon \textbf{a}\to \C$; we need to show that $\Phi_\textbf{a}(\alpha) = F\alpha$. It is clear that the $\V$-functors $FH_i$ and $\Phi_\textbf{a}(H_i)$ are equal for $i = 1,2$. By naturality of $\Phi$, we have $\Phi_\textbf{a}(\alpha)*v_i = \Phi_\textbf{1}(\alpha*v_i)= F_{H_1i,H_2i}\circ \alpha_i= (F\alpha)*v_i$. Thus $\Phi_\textbf{a}(\alpha) = F\alpha$. Therefore, we have shown the needed equivalence.
\end{proof}

\section*{Declarations}
\textbf{Funding} \\
\noindent
This research was supported by the Fonds de la Recherche Scientifique, FNRS (Belgium).

\smallskip
\noindent\textbf{Declaration of competing interest} \\
The author declares that there are no competing interests.

\smallskip
\noindent\textbf{Data availability}\\
A Lean 4 formalization of the paper's results has been made public in \cite{forsman_syntactic_2026}.

\smallskip
\noindent\textbf{Declaration of generative AI and AI-assisted technologies in the writing process} \\
During the preparation of this work, the author used Anthropic's language model Claude Opus 5.5 to proofread the manuscript, check references, and comment on drafts of abstracts and introductions. After using this tool, the author reviewed the feedback and edited the content as they saw appropriate. The text, the new proofs and the ideas are the author's.

\end{document}